\documentclass[a4paper,12pt]{article}
\usepackage{mathrsfs} 
\usepackage{soul}
\usepackage{amsmath, amssymb, amsfonts, amsthm}
\usepackage{ulem}
\usepackage{authblk}
\usepackage{mathrsfs}
\usepackage{enumitem}
\usepackage{stmaryrd}
\usepackage[latin1]{inputenc}
\usepackage[T1]{fontenc}
\usepackage{color}
\usepackage{graphicx}
\usepackage{bm}
\usepackage{hyperref}
\usepackage{wrapfig}
\usepackage{cleveref}

\theoremstyle{plain}
\newtheorem{theorem}{Theorem}[section]

\newtheorem{lemma}[theorem]{Lemma}
\newtheorem{corollary}[theorem]{Corollary}
\newtheorem{fact}[theorem]{Fact}
\theoremstyle{remark}
\newtheorem{remark}[theorem]{Remark}
\usepackage[numbers]{natbib}
\usepackage{mathtools}
\numberwithin{equation}{section}

\DeclarePairedDelimiterX\intff[2]{[}{]}{#1,#2}
\DeclarePairedDelimiterX\intfo[2]{[}{)}{#1,#2}
\DeclarePairedDelimiterX\intof[2]{(}{]}{#1,#2}
\DeclarePairedDelimiterX\intoo[2]{(}{)}{#1,#2}
\DeclarePairedDelimiter{\pars}{(}{)}
\DeclarePairedDelimiter{\bracks}{[}{]}

\DeclarePairedDelimiter{\braces}{\lbrace}{\rbrace}
\DeclarePairedDelimiterX{\setof}[2]{\lbrace}{\rbrace}{#1\,{:}\,#2}
\DeclarePairedDelimiterX{\bracksof}[2]{[}{]}{#1\,\delimsize\vert\,#2}
\DeclarePairedDelimiterX{\parsof}[2]{(}{)}{#1\,\delimsize\vert\,#2}
\DeclarePairedDelimiterXPP\lnorm[2]{}\lVert\rVert{_{#1}}{#2}

\newcommand{\capc}{ \mathtt{cap}^{ (c) } }
\newcommand{\capd}{\mathtt{cap}^{(d)}}
\newcommand{\greenc}{g}
\newcommand{\greend}{G^{({d})}}
\newcommand{\Pc}{{\rm P}^{({\tt BM})}}
\newcommand{\Pd}{{\rm P}^{({\tt S})}}
\newcommand{\Pmodel}{\mathbb P}
\newcommand{\Emodel}{\mathbb E}  
\newcommand{\ball}{\mathtt{Ball}}
\newcommand{\dist}{\mathtt{d}}
\newcommand{\Pbrw}{\mathbf P^{(\tree_\infty)}}

\newcommand{\Ebrw}{\mathbf E^{(\tree_\infty)}}

\def\r{{\mathbb R}}
\def\e{{\mathbb E}}
\def\p{{\mathbb P}}

\def\P{{\bf P}}
\def\E{{\bf E}}

\def\z{{\mathbb Z}}

\def\RR{{\mathfrak R}}
\def\hit{{T}}
\def\A{{\mathscr A}}

\newcommand{\tree}{{\mathcal T}}
\def\law{{\buildrel \mbox{\tiny\rm (law)} \over =}}
\def\wcv{{\buildrel \mbox{\tiny\rm (law)} \over \longrightarrow}}
\def\new{*}

\usepackage[left=2.5cm,right=2.5cm,vmargin=2.5cm]{geometry}

\title{Convergence in law for the capacity of the range of a critical branching random walk}

\author{Tianyi Bai\footnote{\scriptsize New York University Shanghai, 1555 Shiji Blvd, Pudong, Shanghai, Chine, 200122. Email: tianyi.bai73@gmail.com}\hskip10pt  and Yueyun Hu\footnote{\scriptsize LAGA, Universit\'e Paris XIII, 99 av. J.B. Cl\'ement, 93430 Villetaneuse cedex, France. Email: yueyun@math.univ-paris13.fr}}

\date{}

\begin{document}

\maketitle

\begin{abstract}
Let $R_n$ be the range of a critical branching random walk 
with $n$ particles on $\mathbb Z^d$, which is the set of sites visited by a random 
walk indexed by a critical Galton--Watson tree conditioned on having exactly 
$n$ vertices.  For  $d\in\{3, 4, 5\}$, we prove that $n^{-\frac{d-2}4} 
\mathtt{cap}^{(d)}(R_n)$, the renormalized capacity of $R_n$, converges in law to the 
capacity of {the support of the integrated super-Brownian excursion}. The 
proof relies on a study of the intersection probabilities between the critical 
branching random walk and an independent simple random walk on $\mathbb Z^d$.
\end{abstract}

\section{Introduction}

Let $\theta$ be a centered probability distribution on $\z^d$.  For any discrete planar tree $\tree$ rooted at $\varnothing$, we may define a $\z^d$-valued random walk $V_\tree\equiv(V_\tree(u))_{u\in \tree}$ as follows: To all edges $e$ of $\tree$ we associate i.i.d. random variables $X(e)$ with common distribution $\theta$. Let $V_\tree(\varnothing):=0$. For any $u \neq \varnothing$, let $V_\tree(u)$ be the sum of  $X(e)$ for those edges $e$ belonging to the simple path in $\tree$ relating $\varnothing$ to $u$. We also call $V_\tree$ a branching random walk (BRW) indexed by $\tree$.

In this paper we take $\tree$ to be the genealogical tree of a critical 
Galton-Watson process with offspring distribution $(p_i)_{i\ge 0}$ (critical 
means $\sum_{i\ge 0}ip_i=1$) and starting with one single individual. Denote by 
$\#\tree$ the number of  vertices of $\tree$ which is almost surely finite.  
Let $\tree^{(n)}$ be $\tree$ conditioned by $\{\#\tree=n\}$ (we consider in the 
sequel only those $n$ such that $\p(\#\tree=n)>0$).  Then $ V_{\tree^{(n)}} $ 
is a  BRW indexed by the critical Galton--Watson tree $\tree$ conditioned on 
having exactly $n$ vertices.

We are interested in the range $R_n$ of $V_{\tree^{(n)}}$, which is the set of sites visited by $V_{\tree^{(n)}}(u)$ when $u$ runs through the whole tree $\tree^{(n)}$: $$R_n:=\{V_{\tree^{(n)}}(u), u \in \tree^{(n)}\} \subset \z^d.  $$

Denote by $\#R_n$ the cardinality  of $R_n$.  Under mild assumptions on $(p_i)$ and $\theta$,  Le Gall and Lin \cite{LeGall-Lin-lowdim, LeGall-Lin-range} have obtained precise asymptotic behaviour of $\#R_n$ for all dimensions  (see also Lin \cite{Lin-drift} for the case when $\theta$ is not centered):  \begin{equation}\label{LG-Lin}\begin{cases}
\frac{1}{n} \#R_n {\buildrel \mbox{\tiny\rm (p)} \over \longrightarrow} \, c_{\theta, p, d} , \qquad & \mbox{if $d\ge 5$}, \\
\frac{\log n}{n} \#R_n {\buildrel \mbox{\tiny\rm ($L^2$)} \over \longrightarrow}\, c_{\theta, p, d} , \qquad & \mbox{if $d=4$}, \\
n^{-d/4} \#R_n \wcv \,   \lambda_d(\RR), \qquad & \mbox{if $d\le 3$,}
\end{cases}
\end{equation}

\noindent where $c_{\theta, p, d} $ is some positive constant,  $\lambda_d$ 
denotes the Lebesgue measure on $\r^d$, and  $\RR$  stands for the support of 
the  (rescaled)  integrated super-Brownian excursion (ISE) and can be 
realized as in \eqref{def-RR} below.   One may interpret $d=4$ as the critical 
dimension for the cardinality of $R_n$.

We study here the capacity of $R_n$ and assume $d\ge 3$. For any finite subset $A\subseteq\mathbb Z^d$,   its (discrete) capacity $\capd(A)$ is defined as  \begin{equation}\label{def-capd} \capd (A)=\sum_{x\in A}\Pd_x(\tau_A^+=\infty), \end{equation}

\noindent where $\tau_A^+:= \inf\{n\ge 1: S_n \in A\}$ and $\Pd_x$ denotes the law of a $\z^d$-valued simple random walk (SRW) $(S_n)_{n\ge 0}$ started at $x$.

The capacity of the range of a random process heavily depends on its geometry.  For a SRW on $\z^d$, there is a systematic study by Asselah, Schapira and Sousi (see \cite{ASS18, ASS19}  for further references and motivations from the random interlacements). In particular it was shown in Asselah and Schapira \cite{AS20} an interesting relationship between the deviations of the capacity of the range and the folding phenomenon of a random walk. We note in passage that $d=4$ is the critical dimension for the capacity of the range of a SRW. Moreover, there are also recent studies of capacity for loop-erased random walks motivated by properties of uniform spanning trees, see Hutchcroft and Sousi \cite{HS}.

For the critical BRW, it was proved in \cite{bai2020capacity} that $\capd (R_n)$ satisfies a law of large numbers if $d\ge 7$ and behaves as $\frac{n}{\log n}$ if $d=6$. In \cite{BaiHu}, we showed that when $d\in \{3, 4, 5\}$, $\capd (R_n)= n^{\frac{d-2}{4} +o(1)}$ in probability and therefore confirmed that $d=6$ is the critical dimension for $\capd (R_n)$. 

The main goal of this paper is to study the scaling limits of $\capd (R_n)$ in low dimensions $d\in \{3, 4, 5\}$.  Assume from now on that

\begin{equation}\label{hyp-theta}
\begin{aligned}
& \text{$\theta$ is not supported by any strict subgroup of $\z^d$, $\theta$ is symmetric and}\\
& \e (|X|^q) <\infty,\quad \text{for some fixed $q>4$} ,
\end{aligned}
\end{equation}

\noindent  where $X$ is a random variable distributed as $\theta$. Let $\Sigma_\theta$ be the unique positive definite matrix such that $\Sigma_\theta^2$ is equal to the covariance matrix of $X$.  For the offspring distribution $(p_i)$ of the Galton--Watson tree $\tree$, we assume that  
\begin{equation}\label{hyp-GW} \sum_{i=0}^\infty i p_i=1, \qquad \sigma_p^2:= \sum_{i=0}^\infty i^2 p_i -1 \in (0, \infty).  
\end{equation}

Let us recall a result of Janson and Marckert (\cite{JM05}) on the convergence of the renormalised discrete snake   (see Marzouk \cite{Marzouk} for the optimal assumptions on $\theta$ and $(p_i)$). Denote by $\{{\tt w}_k, 0\le k \le 2 (n-1)\}$  the contour walk on $\mathcal T^{(n)}$.  Let $({\bf r}_n(t))_{0\le t\le 1}$ be the linear interpolation of $ ((n-1)^{-1/4} V({\tt w}_{\lfloor 2(n-1)  t\rfloor}))_{0\le t \le 1}$ which are the normalised spatial positions of these vertices, where $\lfloor s \rfloor$ denotes the integer part of $s\in \r_+$.
Then \begin{equation}\label{JM05} 
({\bf r}_n(t))_{0\le t \le 1} \, \wcv \, \Big((\frac{2}{\sigma_p})^{1/2}\, 
\Sigma_\theta\, {\bf r}(t)\Big)_{0\le t\le 1}, 
\end{equation}

\noindent where the convergence holds in the space of continuous functions ${\cal C}[0, 1]$ endowed with the sup-norm, and ${\bf r}$ stands for the Brownian snake: conditionally on the normalised Brownian excursion ${\bf e}=({\bf e}(t), 0\le t \le 1)$, ${\bf r}$ is a centered Gaussian process with covariance matrix   $$ {\rm Cov}({\bf r}(s), {\bf r}(t) | {\bf e})= \min_{s\le u\le t} {\bf e}(u) \, I_d, \qquad 0\le s\le t\le 1,$$

\noindent with $I_d$ the identity matrix $d\times d$.  
Let \begin{equation}\label{def-RR} \RR:= \Big\{(\frac{2}{\sigma_p})^{1/2} \, \Sigma_\theta\,{\bf r}(t), 0\le t \le 1\Big\}
\end{equation}

\noindent be the support of the integrated super-Brownian excursion (ISE) 
(rescaled by the factor $(\frac 2 {\sigma_p})^{1/2}\Sigma_\theta$). We 
refer to Le Gall \cite{LeGall1999} for further properties on the Brownian snake and ISE.  

The first result  of this paper is 

\begin{theorem}\label{thm:convergence0} Assume \eqref{hyp-theta} and \eqref{hyp-GW}.
	In dimensions $d=3,4,5$, as $n\rightarrow\infty$,
	\[
	n^{-\frac{d-2}{4}}\capd (R_n)\, \wcv\, \frac 1 d\,  \capc (\RR) ,
	\]
	where $\capc (\RR)$ denotes the  Newtonian capacity of $\RR$, see \eqref{def-capc} 
	for the definition. 
\end{theorem}

\begin{remark}
	(i)  By Delmas \cite{Delmas99}, a.s. $\capc (\RR) >0$ for $d\in \{3, 4, 5\}$ 
	whereas $\capc (\RR) =0$ for $d\ge 6$.  Indeed, the cases $d\ge 5$ and $d=3$ 
	follow from Proposition 4.3 and Lemma 4.4  of   \cite{Delmas99}, whereas the 
	case $d=4$ can be obtained from a modified version of Lemma 4.5 there. For 
	$d=4$, by imitating the arguments in the proof of Lemma 4.5 we can show that  
	$\mathbb N_x[S_\varepsilon(\int_0^T d s Y_s)]=O(\log 1/\varepsilon)$.  It follows by monotonicity (see the end of Lemma 4.5 there) and Borel-Cantelli's lemma that  
	for any $\alpha>1$,  
	$N_x$-a.e., $S_\varepsilon(\int_0^T d s Y_s)= 
	O((\log 1/\varepsilon)^\alpha)$ as $\varepsilon\to 0$, from which we deduce 
	that $\capc (\RR)>0$ for $d=4$.   The fact that $\capc (\RR)=0$ for 
	$d\ge 6$ is in accordance with the asymptotic behaviors of $\capd (R_n)$ in 
	\cite{bai2020capacity}.
	
	(ii) The dependencies on $\sigma_p$ and $\Sigma_\theta$ are hidden in the 
	definition of $\RR$, and 
	the factor $\frac 1 d$ comes from the difference 
	between Green's functions of $\z^d$ and $\r^d$ (see \eqref{eq:Gdg}).

	(iii) The integrability of $X\law\theta$ in \eqref{hyp-theta} is nearly optimal 
	in the case of finite variance of $(p_i)$, in fact  as shown in \cite{JM05} and 
	\cite{Marzouk}, the optimal integrability on $X$ to ensure 
	\eqref{JM05} is $ \p(|X|\ge s)= o(s^{-4})$ as $s\to \infty$. We need $q>4$ in \eqref{hyp-theta} to get some H\"{o}lder continuity on the BRW,  see \Cref{l:modulo}.
	The symmetry of $\theta$ is to guarantee that \eqref{invariance} and 
	\eqref{absolutecontinuity1} hold at the same time, see \cite[Remark 
	2.5]{bai2020capacity} for explanation.
	
	(iv) When $|X|$ has a regular varying tail of exponent $4$, the BRW converges to the so-called jumping snake, see \cite[Theorem 5]{JM05} which is still sufficient to provide \Cref{cor:epsilon} (with Brownian snake replaced by jumping snake). This motivates us to conjecture that the convergence of capacity remains valid as well.
	\hfill$\Box$ 
\end{remark}

Let us say a few words on the proof of Theorem \ref{thm:convergence0}. 
By the Skorokhod representation theorem, there is a probability space on which the convergence in \eqref{JM05} holds almost surely. We shall work on this (possibly extended) probability space in the rest of this paper and prove that \begin{equation}\label{t:L1}
n^{-\frac{d-2}{4}}\capd (R_n)\,   {\buildrel \mbox{\tiny $\rm(p)$} \over 
	\longrightarrow}   \,  \frac 1 d\, \capc (\RR). \end{equation}

While the upper bound in \eqref{t:L1} is essentially a consequence of the almost sure version of \eqref{JM05}, the lower bound is more delicate and relies heavily on the study of intersection probabilities between the BRW and an independent SRW. Let for $x \in \r^d$ and $A\subset \r^d$, $\dist(x, A):= \min_{y\in A} |x-y|$ be the distance between $x$ and  the set $A$. Denote by $$\tau_A:= \inf\{n\ge 0: S_n \in A\} \in {\mathbb N} \cup\{\infty\} $$  the first hitting time to $A$ by the SRW $(S_n)$. 
The following result plays a crucial role in the proof of the lower bound in \eqref{t:L1} and may be of independent interest:

\begin{theorem}\label{lem:inter1}
	Assume \eqref{hyp-theta} and \eqref{hyp-GW}. In dimensions $d=3,4,5$,
	\begin{equation}\label{eq:inter1}
	\lim_{\lambda\rightarrow 0+}\limsup_{n\rightarrow\infty}\Emodel\Big[
	\sup_{\dist(x,R_n)<\lambda n^{1/4}} \Pd_x(\tau_{R_n}=\infty) \Big]
	=0,
	\end{equation}
	where  under $\Pd_x$ we compute the probability only with respect to the SRW $(S_n)$. 
\end{theorem}

\begin{remark}
We remark that the SRW in Theorem \ref{lem:inter1} can actually be replaced by any random walk with symmetric, bounded and irreducible displacement following the same proof in Section \ref{s:intersection} with minor modifications. 
\end{remark}

\begin{remark}
	In higher dimensions, the non-intersection probability between SRW and BRW goes 
	to $1$ if we start the SRW relatively far from the origin: 
	Assume \eqref{hyp-theta} and \eqref{hyp-GW}, if $d\ge 6$, then for any 
	$\lambda>0$, $$
	\sup_{\dist(x,R_n)<\lambda n^{1/4}} \Pd_x(\tau_{R_n}=\infty) \,   {\buildrel \mbox{\tiny\rm $(p)$} \over \longrightarrow}   \,  1, \qquad n\to\infty.$$
	This is an easy consequence of the behaviors of $\capd(R_n)$ in \cite{bai2020capacity} by using \eqref{eq:discretenu}. \hfill$\Box$
\end{remark}

\medskip
In the study of the probability term in \eqref{eq:inter1}, the main 
obstacle is the lack of independence in $V_{\tree^{(n)}}(u)$ when $u$ runs through $\tree^{(n)}$ in   lexicographic order. This will be overcome in Section  \ref{s:optionalline} by using
some optional lines for $V_{\tree_\infty}$ a BRW  indexed by an infinite tree 
$\tree_\infty$. In fact we shall consider two infinite trees: $\tree_\infty$ 
and its counterpart $\tree^*_\infty$.  Introduced in  \cite{LeGall-Lin-range}, 
both $\tree_\infty$ and $\tree^*_\infty$ can be viewed as a family of i.i.d. copies 
of $\tree$ glued in a certain way to  an infinite ray called spine. The infinite tree $\tree^*_\infty$ was constructed so that $V_{\tree^*_\infty}$ satisfies an invariance in law by translation (see \eqref{invariance}). The use of optional lines of $\tree_\infty$ has the 
advantage to better explore the Markov property of the BRW (Lemma 
\ref{lem:4.5}).  Then an iteration argument in Section \ref{s:iteration}, 
inspired from Lawler \cite{Lawler1996},  will give a  (fast enough) 
decay of the 
non-intersection probability (Lemma \ref{lem:epsilonMd2}).  As stated in  Lemma 
\ref{lem:dist_spine}, we may compare $ \tree^*_\infty $ and $ \tree_\infty $, 
then deduce the corresponding result for $V_{\tree^*_\infty}$ (Corollary 
\ref{cor:epsilonMd2-new}). This together with the  stationary increments  of  
$V_{\tree^*_\infty}$  yield an analogue of \eqref{eq:inter1} for 
$V_{\tree^*_\infty}$ as well as $V_{\tree_\infty}$ (Theorem 
\ref{thm:mainfinite}). Finally we use the absolute continuity between 
$V_{\tree_\infty}$ and $V_{\tree^{(n)}}$ established in Zhu \cite{Zhu-cbrw} (see \eqref{absolutecontinuity1})  and prove Theorem \ref{lem:inter1}.

\medskip
The rest of the paper is organised as follows: 

$\bullet$ In Section \ref{s:pre},  we collect some known facts on the discrete and Newtonian  capacities and some preliminary results on the BRW; 

$\bullet$ In Section \ref{s:proofthm1}, we give the proof of Theorem \ref{thm:convergence0} by admitting Theorem \ref{lem:inter1};

$\bullet$ In Section \ref{s:intersection}, we first introduce the two infinite trees $\tree_\infty$ and $\tree^*_\infty$,  then study the  intersection probability between SRW and  the two BRWs $V_{\tree_\infty}$ and $V_{\tree^*_\infty}$. The main result in this Section is Theorem \ref{thm:mainfinite}, from which we deduce Theorem \ref{lem:inter1} in Section  \ref{sec:3.2}.

\medskip
In the sequel, let $(X_k)_{k\ge 0}$,  under $\p$,  be  a random walk on $\z^d$ with step distribution $\theta$, starting from $0$. We shall denote by $C, C', C''$ (eventually with subscripts) some positive constants whose values may change from one paragraph to another, and  by $\Pd_x$, $\Pc_x$, the law  of a SRW $(S_n)$ on $\z^d$ and that of a standard Brownian motion $(W_t)$ in $\r^d$, started at $x$.    
For notational brevity, we consider parameters (e.g. $\varepsilon n$) as they were integers in expressions like $S_{\varepsilon n}$.

\section{Preliminaries}\label{s:pre}
\subsection{Discrete capacity in \texorpdfstring{$\z^d$}{}}

Let $A\subset \z^d$ be a finite set.  By the Markov property of SRW, we have that for any $x\in\mathbb Z^d$,
\begin{equation}\label{eq:discretenu}
\sum_{y\in A} \greend(x,y)\Pd_y(\tau^+_{A}=\infty)=\Pd_x(\tau_{A}<\infty),
\end{equation}

\noindent where $\greend$ is the Green function for the SRW $(S_n)$: 
$\greend(x,y):= \greend(y-x)$ and 
\begin{equation}\label{Gd-asymp}\greend(x):= \sum_{n=0}^\infty \Pd(S_n=x) = 
c_1\, |x|^{2-d} +O( |x|^{1-d}), \qquad |x|\to \infty, \end{equation}

\noindent 
with $c_1:= \frac{d \Gamma(\frac{d}2 -1)}{2 \pi^{d/2}} $. Recall 
\eqref{def-capd}, let $|x|\rightarrow\infty$ in \eqref{eq:discretenu}, then 
\begin{equation}\label{eq:discrete_capacity_lim_def}
\capd (A)=\lim_{|x|\rightarrow\infty}\frac{1}{\greend(x)}\Pd_x(\tau_A<\infty).
\end{equation}

By Lawler and Limic (\cite{Lawler-Limic}, Proposition 6.5.1), there exists some $C>0$ such that for any $A\subset \z^d$ and all  $x\in \z^d$ with $|x| \ge 2 \max|A|$,  \begin{equation}\label{eq:lawler-limic} \Big| \capd (A)- \frac{1}{\greend(x)}\Pd_x(\tau_A<\infty)\Big| \le  C\,   \capd (A)  \frac{\max|A|}{|x|} ,
\end{equation}

\noindent where $\max|A|:= \max_{a\in A} |a|$.  
\subsection{Newtonian capacity in \texorpdfstring{$\r^d$}{}}

Let $B \subset \r^d$ be a bounded $F_\sigma$ set (countable union of 
compact sets). The Newtonian capacity  of $B$ is determined by its equilibrium measure 
$\mu_B$ as follows:  For any positive measure $\nu$ in $\r^d$, let  $$ 
g\ast\nu(x):= \int_{\r^d} g(x, y) \nu(dy), \qquad x\in \r^d, $$

\noindent where $g$ denotes the Green function of the standard Brownian motion in $\r^d$: \begin{equation}\label{def-g}
\greenc(x,y)=g(x-y):=\frac{\Gamma(d/2-1)}{2\pi^{d/2}} |x-y|^{2-d}, \qquad x\ne 
y,\,x,y \in \r^d.
\end{equation}

By Port and Stone \cite[Theorem 3.1.10]{bmclassical}, 
there exists a unique measure $\mu_B$, called the equilibrium measure for 
$B$, supported on regular points of $B$  such that
\begin{equation}\label{def-muB}
g\ast \mu_B(x)=1,\qquad \forall x\in B.
\end{equation}

The Newtonian capacity of $B$ is then by definition the total mass of $\mu_B$: \begin{equation}\label{def-capc} \capc(B):=\mu_B(B).\end{equation}

\noindent
By \cite[Theorem 3.1.10]{bmclassical}, \begin{equation}\label{eq:CapToGc}
g\ast \mu_B(x)=\Pc_x(\hit_B<\infty), \qquad \forall x\in\mathbb R^d,
\end{equation}

\noindent where \begin{equation}\label{def-TA} \hit_B:= \inf\{t\ge0: W_t \in B\}
\end{equation}

\noindent denotes the first entrance time to $B$ by a $d$-dimensional Brownian motion $(W_t)_{t\ge 0}$ starting from $x$ (under $\Pc_x$). 
Moreover, for any $B \subset \ball(r):=\{x\in \r^d: |x|\le r\}$ for some 
$r>0$, \begin{equation}\label{eq:port-stone} \capc (B)= \int 
\Pc_x(\hit_B<\infty) d \mu_{\ball(r)}(x), \end{equation}

\noindent where $\mu_{\ball(r)}$ is the equilibrium measure on $\ball(r)$:  $$ 
\mu_{\ball(r)} = \frac{2 \pi^{d/2} r^{d-2}}{\Gamma(\frac{d}{2}-1)} {\mathfrak 
	U}_r= 
\frac{d}{c_1} r^{d-2} {\mathfrak U}_r ,$$

\noindent with ${\mathfrak U}_r$ the uniform probability measure on the sphere 
$\partial \ball(r)$.  Furthermore, we have 
\begin{equation}\label{eq:continuous_capacity_lim_def}
\capc (B)=\lim_{|x|\rightarrow\infty}\frac{1}{\greenc(x)}\Pc_x(\hit_B<\infty).
\end{equation}

The following lemma shows that the Newtonian potential $g\ast \mu$ captures useful information about capacity.
\begin{lemma}\label{lm:converge}
	Let $\mu,(\mu_n)$ be positive $\sigma$-finite measures on $\mathbb R^d,d\ge 3$. If, for any $x\in \mathbb R^d$, we have
	\[
	\liminf_{n\rightarrow\infty} g\ast\mu_n(x)\ge  g\ast\mu(x),
	\]
	then
	\[
	\liminf_{n\rightarrow\infty}\mu_n(\mathbb R^d)\ge\mu(\mathbb R^d).
	\]
\end{lemma}
\begin{proof} Let $r_0>0$. 
	Denote by $\mu_{\ball(r_0)}$ the equilibrium measure on 
	$\ball(r_0)$, then it is 
	supported on $\partial \ball(r_0)$. Applying \eqref{eq:CapToGc}  to 
	$B=\ball(r_0)$, 
	we deduce from Fubini's theorem that $$ \int_{\r^d} g\ast\mu_n (x) 
	\mu_{\ball(r_0)}( dx) = \int_{\mathbb 
		R^d}\mu_n(dy)\Pc_y(\hit_{\ball(r_0)}<\infty),$$
	
	\noindent the same holds for $\mu$ in lieu of $\mu_n$.  By assumption on 
	$g\ast\mu_n$ and Fatou's lemma,  we have $$  \liminf_{n\to\infty} \int_{\r^d} 
	g\ast\mu_n (x) \mu_{\ball(r_0)}( dx) \ge \int_{\r^d} g\ast\mu  (x) 
	\mu_{\ball(r_0)}( dx).$$

	\noindent Therefore,
	\begin{align*}
	\liminf_{n\rightarrow\infty}\mu_n(\mathbb R^d)
	\ge&\liminf_{n\rightarrow\infty}\int_{\mathbb 
		R^d}\mu_n(dy)\Pc_y(\hit_{\ball(r_0)}<\infty)\\
	\ge&\int_{\mathbb R^d}\mu(dy)\Pc_y(\hit_{\ball(r_0)}<\infty)
	\ge\mu(\ball(r_0)).
	\end{align*}
	
	\noindent The Lemma follows from the monotone convergence theorem by letting $r_0\uparrow\infty$.
\end{proof}

Now we recall some known facts.  To begin with, we need the following multi-dimensional extension of the classical Koml\'os-Major-Tusn\'ady coupling between random walks and Brownian 
motion:

\begin{fact} [Einmahl \cite{einmahl1989extensions}]  On a suitable probability space we may construct a   simple random walk $(S_k)$ on $\z^d$ and a standard Brownian motion $(W_t)$ in $\r^d$,  such that for some positive constant  $C$ and for all $j\ge 1$ and $t>0$, \begin{equation} \label{KMT}
	\p\Big( \max_{0\le k \le j} |S_k - d^{-1/2} W_k| \ge t\Big) \le  C \, j \, e^{- t^{1/2}/C}.
	\end{equation}
\end{fact}

We assume in the sequel that \eqref{KMT} and the almost sure convergence of 
\eqref{JM05} simultaneously hold on a common probability space $(\Omega, 
{\mathscr F}, \p)$.

\medskip
For any $A\subseteq\mathbb R^d$ and $r>0$, let  $$A^r:=\{x\in \r^d: \dist(x, A) \le r\}$$

\noindent be the closed $r$-neighborhood of $A$. The almost sure convergence of \eqref{JM05} yields 
\begin{corollary}\label{cor:epsilon} Assume \eqref{hyp-theta} and \eqref{hyp-GW}.
	For any $\varepsilon>0$, $\p$-almost surely for all large $n$, we have \[
	\  n^{-\frac{1}{4}}R_n\subseteq  \RR^{\varepsilon} \quad \mbox{and} \quad \RR\subseteq \pars*{   n^{-\frac{1}{4}}R_n}^{\varepsilon}. 
	\]
\end{corollary}

We end this section by the following estimate: Let $(X_k)_{k\ge 0}$,  under $\p$,  be  a random walk on $\z^d$ with step distribution $\theta$ and $X_0=0$.  By the finite $q$-th moment in \eqref{hyp-theta}, applying Doob's maximal inequality  and Petrov (\cite{Petrov}, Theorem 2.10) we get  that \begin{equation}\label{petrov-moment}\e(\max_{1\le i\le k}|X_i|^q) \le C  k^{q/2}, \qquad \forall\, k\ge 1. \end{equation}

\section{Proof of Theorem \ref{thm:convergence0} by admitting Theorem 
	\ref{lem:inter1}}\label{s:proofthm1}

Recall that on $(\Omega, {\mathscr F}, \p)$, both \eqref{KMT} and the almost 
sure convergence of \eqref{JM05} hold simultaneously. We admit Theorem 
\ref{lem:inter1} and prove \eqref{t:L1}, which obviously yields Theorem 
\ref{thm:convergence0}.  

For any $K>0$, let $$\A_{n, K}:= \{\sup_{0\le t\le 1} |{\bf r}_n(t)| \le K  \}, \qquad \A_{\infty, K}:= \{\sup_{y\in \RR}|y| \le K\}.$$

\noindent Since $\p$-a.s., $\sup_{0\le t\le 1} |{\bf r}_n(t)|  \to \sup_{y\in 
	\RR}|y|$ which is finite, $\limsup_{n\to\infty} \p(\A_{n, K}^c) \to 0$ as 
$K\to\infty$.  Moreover, notice that except for at most countably many $K$, we have 
$\p(\sup_{y\in \RR}|y|=K)=0$. In particular, it is not hard to see that for 
those $K$,
\begin{align}
&\limsup_{n\rightarrow\infty}\A_{n,K}\subseteq\A_{\infty,K},\quad\p\text{-almost
	surely};\label{eq:k1}\\
&\p(\A_{\infty,K}\cap\A^c_{n,K})\rightarrow 0,\quad 
n\rightarrow\infty.\label{eq:k2}
\end{align}

Then to get \eqref{t:L1}, it suffices to show that for any fixed 
$K>0$ such that $\p(\sup_{y\in \RR}|y|=K)=0$, \begin{equation}\label{t:L1-bis}
n^{-\frac{d-2}{4}}\capd (R_n)\, 1_{\A_{n,K}}   {\buildrel \mbox{\tiny\rm 
		$L^1$} \over 
	\longrightarrow}   \,  \frac 1 d\, \capc (\RR) \, 1_{\A_{\infty,K}}. 
\end{equation}

The proof of \eqref{t:L1-bis} is mainly outlined by the following 
\Cref{fact-cv} with
$$\xi_n = n^{-\frac{d-2}{4}}\capd 
(R_n) \, 1_{\A_{n,K}}, \qquad \xi =\frac 1 d \, \capc (\RR)\, 
1_{\A_{\infty,K}}.$$

\begin{fact}\label{fact-cv} Let $(\xi_n)_{n\ge 1}$ be a family of uniformly 
	integrable nonnegative random variables. Assume that for some random variable 
	$\xi$ we have 
	
	(i) $\limsup_{n\to\infty} \xi_n \le \xi$ almost surely;
	
	(ii) $\liminf_{n\to\infty} \e(\xi_n) \ge \e(\xi).$
	
	\noindent Then $\lim_{n\to\infty} \e(|\xi_n-\xi|)=0.$
\end{fact}

\begin{proof}[Proof of \Cref{fact-cv}]
	Note that $\e(\xi)<\infty$   and  
	$\e(|\xi_n-\xi|)= 2 \e((\xi_n-\xi)^+)- \e(\xi_n-\xi)$. By (i), 
	$(\xi_n-\xi)^+\to0$ almost surely, we deduce from the uniform integrability 
	that $\e((\xi_n-\xi)^+)\to0$, which in view of (ii) implies that 
	$\limsup_{n\to\infty} \e(|\xi_n-\xi|)=0$. 
\end{proof}
In fact, the (discrete) capacity of a ball in $\z^d$, 
centered at the origin and with radius $r$, is less than  $C_d \, r^{d-2}$ for 
any $r\ge 1$, we have  
\[ \capd (R_n) \le  C_d\, (\max_{x\in R_n} |x|)^{d-2}.
\]

\noindent It follows that $\xi_n \le C_d K$ for any $n$, hence $(\xi_n)_{n\ge 1}$ is uniformly integrable.   To get \eqref{t:L1-bis}, we shall prove \begin{eqnarray}
&&\limsup_{n\to\infty} n^{-\frac{d-2}{4}}\capd 
(R_n) \le \frac 1 d\, \capc (\RR), \qquad \mbox{a.s.}, \label{xi-upper}
\\
&& \liminf_{n\to\infty} \e(\xi_n) \ge \e\Big(\frac 1 d \, \capc (\RR)\, 
1_{\A_{\infty,K}}\Big). \label{xi-lower}
\end{eqnarray}

\noindent 
Indeed, \eqref{eq:k1} and \eqref{xi-upper} imply the condition (i) in 
\Cref{fact-cv}, and we may apply \Cref{fact-cv} to obtain 
\eqref{t:L1-bis}.

We check  \eqref{xi-upper} and \eqref{xi-lower} in the following two 
subsections respectively.

\subsection{Upper bound: proof of \texorpdfstring{\eqref{xi-upper}}{}}\label{sec:upper}

To begin with,  we have 
\begin{lemma} \label{l:capc-pre} Assume \eqref{hyp-theta} and \eqref{hyp-GW}. Let $d\ge 3$. 
	For any $\varepsilon\in(0,\frac 1 4)$,  
	\[\limsup_{n\rightarrow\infty}  n^{-\frac{d-2}{4}}\capc (R_n^{n^\varepsilon})\le \capc (\RR),\qquad\Pmodel\text{-almost surely},\]
	where we recall that $R_n^{n^\varepsilon}$ denotes the closed $n^\varepsilon$-neighborhood of $R_n$ in $\r^d$. \end{lemma}

It is necessary to use neighborhoods instead of the exact ranges on the left-hand side of the above inequality, otherwise its Newtonian capacity is always trivially $0$ in dimensions $d\ge 3$.

\begin{proof}
	Let $\varepsilon'>0$. By \Cref{cor:epsilon}, $\Pmodel$\text{-almost surely} for all $n$ large, we have \[
	n^{-\frac{1}{4}} R_n^{n^\varepsilon}\subseteq \RR^{\varepsilon'}, \]
	
	\noindent which implies that $$ \capc (R_n^{n^\varepsilon}) \le \capc(  n^{ \frac{1}{4}} \RR^{\varepsilon'})=   n^{ \frac{d-2}{4}} \capc  (\RR^{\varepsilon'}).$$
	
	By \cite[Proposition 3.1.13]{bmclassical}, $ \lim_{\varepsilon'\rightarrow 0+}\capc (\RR^{\varepsilon'})=\capc (\RR),$    the Lemma follows.  
\end{proof}

{\noindent\it Proof of  \eqref{xi-upper}.} By Lemma \ref{l:capc-pre}, it is enough to show that  $\Pmodel$-almost surely,
\begin{equation}\label{discreteIntersection_to_continuous}
\limsup_{n\rightarrow\infty} n^{-\frac{d-2}{4}}\capd (R_n)\le  \,  \frac 1 d 
\,\limsup_{n\rightarrow\infty} n^{-\frac{d-2}{4}}\capc (R_n^{n^\varepsilon}).
\end{equation}

Since $n^{-1/4} \max_{x\in R_n} |x| \to \sup_{y\in \RR}|y|$ almost surely,  
$\max_{x\in R_n} |x|  \le n^{3/8}$ for all $n$ large enough.

Let $\varepsilon \in (0, \frac18)$ and  $n$ large enough. 
Then $R_n^{n^\varepsilon} \subset \ball({n^{1/2}})$ (the ball in $\r^d$ of 
radius $n^{1/2}$ and centered at $0$). 
For any $x\in \r^d$, let $[x]\in \z^d$ be such that $| 
x-[x]| \le 1$ (if there are several such points $[x]$, we choose an arbitrary 
one).  By \eqref{eq:lawler-limic}, for any $|x|= n$, 
\begin{equation}\label{psdx1} \capd(R_n) \le  \frac{1+o(1)}{\greend([d^{-1/2} 
	x])}\Pd_{[d^{-1/2} x]}(\tau_{R_n}<\infty), \end{equation}

\noindent where as before, under $\Pd$ we compute the probability only with respect to the SRW $(S_n)$.  By \eqref{Gd-asymp},   $$ \greend([d^{-1/2} x]) = (c_1 d^{(d-2)/2} +o(1))\, n^{2-d},$$

\noindent with $o(1) \to 0$ as $n\to\infty$ uniformly in $|x|= n$.  Since $\max_{x\in R_n} |x|  \le n^{3/8}$, we have that for any  $|x|=  n $, \begin{eqnarray*}\Pd_{[d^{-1/2} x]}(\tau_{R_n}<\infty)  &\le&
	\Pd_{[d^{-1/2} x]}(\tau_{R_n} \le n^6) + \Pd_{[d^{-1/2} x]}(\inf_{j\ge n^6} |S_j| \le n^{3/8})
	\\
	&\le&
	\Pd_{[d^{-1/2} x]}(\tau_{R_n} \le n^6) + \Pd_0(\inf_{j\ge n^6} |S_j| \le 2 n).
\end{eqnarray*}

\noindent For any $r>0$, we have $\Pd_0(\inf_{j\ge n^6} |S_j| \le 2 n) \le 
\Pd_0(|S_{n^6}| \le r) + \sup_{|x|\ge r} \Pd_x(\inf_{j\ge 0} |S_j|\le 2 n)$, 
which by the local limit theorem for the first probability term and  
Proposition 6.4.2 in \cite{Lawler-Limic} for the second, is less than $C r^d 
n^{- 3 d } + C (n/r)^{d-2}.$ Choosing $r=n^2$, we get that  $\Pd(\inf_{j\ge 
	n^6} |S_j| \le 2 n) \le C' n^{2-d}$.  Then we have shown that uniformly in 
$|x|=n$, \begin{equation}\label{psdx2} \Pd_{[d^{-1/2} x]}(\tau_{R_n}<\infty)  
\le
\Pd_{[d^{-1/2} x]}(\tau_{R_n} \le n^6) + C' n^{2-d}.\end{equation}

Using the coupling between the SRW   and the Brownian motion   in \eqref{KMT}, we have that for all large $n$ and  $|x|=  n$,   \begin{eqnarray*}
	\Pd_{[d^{-1/2} x]}(\tau_{R_n} \le n^6)
	&\le&
	\Pc_{x}(\hit_{d^{1/2}\, R_n^{n^\varepsilon}} \le n^6) + \Pmodel \Big( \max_{0\le k \le n^6} | S_k - d^{-1/2} W_k | \ge   \frac{n^\varepsilon}{2}\Big)
	\\ 
	&\le& 
	\Pc_{x}(\hit_{d^{1/2}\, R_n^{n^\varepsilon}} <\infty) +   e^{-n^{\varepsilon/3}}.
\end{eqnarray*}

\noindent In view of \eqref{psdx1} and \eqref{psdx2}, this yields that  $\Pmodel$-almost surely for all large $n$,    $$  \capd(R_n) \le \frac{1+o(1)}{c_1 d^{(d-2)/2}} n^{d-2}  \Pc_x(\hit_{d^{1/2}\, R_n^{n^\varepsilon}} <\infty)+ C'',$$

\noindent where as before, $o(1)\to 0$ as $n\to\infty$  uniformly in $|x|=n$.  
Applying \eqref{eq:port-stone} to $B=R_n^{n^\varepsilon}$ and $r=n$ there, we 
integrate the above inequality with respect to $\mu_{\ball(n)}$ and get that $$ 
\capd(R_n) \le \frac{1+o(1)}{d^{d/2}}  \capc(d^{1/2} R_n^{n^\varepsilon}) + 
C''.$$

\noindent Since $\capc(d^{1/2} R_n^{n^\varepsilon})= d^{(d-2)/2} \capc(R_n^{n^\varepsilon})$,  we get \eqref{discreteIntersection_to_continuous} and hence \eqref{xi-upper}. \hfill$\Box$

\subsection{Lower bound: proof of \texorpdfstring{\eqref{xi-lower}}{} by admitting Theorem \texorpdfstring{\ref{lem:inter1}}{}}\label{sec:lower}

The proof of  \eqref{xi-lower} relies on an application of Lemma 
\ref{lm:converge}.  
To this end,  we shall take $\mu_\RR$ as the equilibrium measure of $\RR$ and 
construct a sequence of finite measures  $(\mu_n)_{n\ge 1}$ such that the 
total mass of $\mu_n$ is a normalised version of $\capd (R_n) 1_{\A_{n, 
		K}}$.   
More specifically, let  \begin{equation}\label{def:nu_n}
\mu_n:= n^{-\frac{d-2}{4}} \sum_{x\in R_n} \Pd_x(\tau^+_{R_n}=\infty) \, 
\delta_{\{n^{-\frac 1 4}x\}}\, 1_{\A_{n, K}},
\end{equation}

\noindent with $\delta_z$ the Dirac measure at $z\in \r^d$. Then we have \begin{equation}\label {eq:nu}
\mu_n(\mathbb R^d)=n^{-\frac {d-2} {4}}\capd (R_n) \, 1_{\A_{n, K}}.
\end{equation}

\noindent We shall apply Lemma \ref{lm:converge} to $\mathbb P \otimes \mu_n$ and 
$\mathbb P \otimes \mu$ with $\mu:=\frac{1}{d}\, \mu_\RR 1_{\A_{\infty, K}} $. 
The following 
Lemma reduces the problem of capacity to that of intersection probability:  

\begin{lemma}\label{lem:DiscreteGreen}  Assume \eqref{hyp-theta} and \eqref{hyp-GW}.  In dimension $d\in  \{3, 4, 5\}$, for any $x\in\mathbb R^d$, $$ d \, \liminf_{n\rightarrow\infty}\Emodel\bracks*{  g\ast \mu_n(x)}
	\ge
	\liminf_{n\rightarrow\infty} \Emodel \bracks*{\Pd_{[n^{\frac 1 4} 
			x]}(\tau^+_{R_n}<\infty) \, 1_{\A_{n, K}}},$$
	
	\noindent where as before,  under $\Pd_z$ we only compute the probability with respect to the SRW  $(S_n)$ starting from $z\in \z^d$. 
\end{lemma}

We mention that the last step in the proof of Lemma \ref{lem:DiscreteGreen} 
requires Theorem \ref{lem:inter1}.

\begin{proof} Fix $x\in\mathbb R^d$.  By \eqref{eq:discretenu} and 
	\eqref{def:nu_n},  
	$$
	n^{\frac{d-2}4} \int_{y\in\mathbb R^d}\greend([n^{\frac 1 4}x], n^{\frac 1 
		4}y)\mu_n(dy)=\Pd_{[ n^{\frac 1 4} x]}(\tau^+_{R_n}<\infty) \,1_{\A_{n, K}}.
	$$
	
	\noindent Then it suffices to show that \begin{equation}
	d\, \liminf_{n\rightarrow\infty}\Emodel\bracks*{\int_{y\in\mathbb R^d}  \greenc(x,y)\mu_n(dy)}
	\ge \liminf_{n\rightarrow\infty}\Emodel\bracks*{
		n^{\frac{d-2}4} \int_{y\in\mathbb R^d}
		\greend([ n^{\frac 1 4}x], n^{\frac 1 4}y)\mu_n(dy)}. \label{eq:gGnu_n}
	\end{equation}
	
	First by \eqref{Gd-asymp} and \eqref{def-g}, 
	\begin{equation}\label{eq:Gdg}
	\greend(x,y)=d\, \greenc(x,y)+O(|x-y|^{1-d}), \qquad x, y \in \z^d.
	\end{equation}
	
	\noindent
	For any $\varepsilon>0$,  we can find $C\equiv C_\varepsilon\ge 1 $ such that whenever $|x-y|\ge C $ and $x, y\in \r^d$,
	\[(d+\varepsilon)\greenc(x,y)\ge\greend([x],y),\]
	
	\noindent where   $\greend([x],y):=0$ if $y \not\in \z^d$.   Then for any $|y-x| \ge   C \, n^{-1/4}$, we have  $ n^{\frac{d-2}4} 
	\greend([n^{\frac 1 4}x], n^{\frac 1 4}y) \le (d+\varepsilon) \greenc(x,y)$ and then  $$n^{\frac {d-2}{4}}\int_{|y-x| \ge   C \, n^{-1/4}}
	\greend([ n^{\frac 1 4}x], n^{\frac 1 4}y)\mu_n(dy)
	\le
	(d+\varepsilon) \int_{y\in\mathbb R^d}  \greenc(x,y)\mu_n(dy).
	$$
	
	To get \eqref{eq:gGnu_n}, it is enough to check that for any $C>1$, as $n \to\infty$, $$n^{\frac{d-2}4}  \Emodel\bracks*{ \int_{|y-x| <   C \, n^{-1/4}}
		\greend([ n^{\frac 1 4}x], n^{\frac 1 4}y)\mu_n(dy)} \to 0.$$
	
	\noindent By definition of $\mu_n$, the above left-hand-side expression is 
	\begin{eqnarray*} \mbox{LHS}&=&\Emodel \Big[\sum_{y\in R_n: |y-  n^{\frac 1 
				4}x| <   C} \greend([ n^{\frac 1 4}x], y)  \Pd_y(\tau^+_{R_n}=\infty) \,  
		1_{\A_{n, K}}\Big]
		\\
		&\le&
		(C+1)^d\, \greend(0,0)\, \e\Big[\sup_{y\in R_n}    \Pd_y(\tau^+_{R_n}=\infty)\Big],
	\end{eqnarray*}
	
	\noindent where the inequality follows from $\greend([ n^{\frac 1 4}x], y)\le \greend(0,0)$ and the fact that there are at most $(C+1)^d$ such points $y$ in the sum.  By Theorem \ref{lem:inter1}, 
	$$\e\Big[\sup_{y\in R_n}    \Pd_y(\tau^+_{R_n}=\infty)\Big]\to0,$$
	
	\noindent which completes the proof of the Lemma. 
\end{proof}
\medskip

Now we are ready to give the proof of \eqref{xi-lower}.

\begin{proof}[Proof of \eqref{xi-lower} by admitting Theorem \ref{lem:inter1}]  
	We claim that it is enough to show the following inequality: For any fixed 
	$x\in \r^d$,
	\begin{eqnarray}\label{eq:intersectionConvergence}
	\liminf_{n\rightarrow\infty}\Emodel\bracks*{\Pd_{[xn^{1/4}]}(\tau^+_{R_n}<\infty)
		\, 1_{\A_{n, K}}}
	&\ge &
	\Emodel\bracks*{\Pc_x(\hit_\RR<\infty) \, 1_{\A_{\infty, K}}} 
	\\
	&=& \Emodel\bracks*{\Pc_{d^{1/2} x}(\hit_{d^{1/2}\RR}<\infty) \,  
		1_{\A_{\infty, K}}}, \nonumber
	\end{eqnarray}
	\noindent where the above equality is a consequence of the Brownian scaling.  Indeed, by \Cref{lem:DiscreteGreen} and 
	\eqref{eq:CapToGc} with $B=\RR$, we have
	$$   \liminf_{n\rightarrow\infty}\Emodel\bracks*{  g\ast \mu_n(x)}
	\ge
	\mathbb E[g\ast\mu(x)],$$
	
	\noindent with $\mu=\frac1{d}\, \mu_\RR 1_{\A_{\infty, K}} $. 
	This together with \Cref{lm:converge} shows that
	\[
	\liminf_{n\rightarrow\infty}\mathbb E[\mu_n(\mathbb R^d)]\ge\mathbb 
	E[\mu(\mathbb R^d)],
	\]
	which is exactly \eqref{xi-lower}, by using \eqref{eq:nu} and the fact that $\mu(\mathbb R^d)=  \frac 1 d \, \capc (\RR)\, 1_{\A_{\infty, K}}$.

	Now it remains to prove \eqref{eq:intersectionConvergence}.
	Fix  $\alpha>0$. 
	For any $\lambda>0$, $\Pmodel$-almost surely for all large   $n$, we have \[
	n^{-\frac{1}{4}}R_n\subseteq  \RR^{\lambda},\qquad \RR\subseteq \pars*{n^{-\frac{1}{4}}R_n}^{\lambda}.
	\]
	
	Let $N=n^{1/2+\alpha}$ be large. Denote by $S[1, k]:=\{S_i, 1\le i \le k\}$ and $W[0, t]=\{W_s, 0\le s\le t\}$ for $k\ge 0$ and $t\ge 0$.  Notice that
	\begin{align*}
	&\mathbb E\bracks*{\Pc_{d^{1/2} x}(\hit_{d^{1/2}\RR}<\infty) 1_{\A_{\infty, K}}}\\
	\le&\mathbb E\bracks*{\Pc_{d^{1/2} x}(W[0, n^{\alpha}]\cap 
		(d^{1/2}\RR)\ne\emptyset) \, 1_{\A_\infty,K}}
	+ \Pc_{d^{1/2} x}(W[ n^{\alpha},\infty)\cap 
	\ball(d^{1/2} K)\ne\emptyset)
	\\
	=&\mathbb E\bracks*{\Pc_{d^{1/2} x}(W[0, n^{\alpha}]\cap 
		(d^{1/2}\RR)\ne\emptyset) \, 1_{\A_{n,K}}}+o(1),
	\end{align*}
	where the last equality follows from \eqref{eq:k2} and the transience of the Brownian motion $W$.

	For the probability term in the above expectation,  we have 
	\begin{align}
	\Pc_{d^{1/2} x}(W[0, n^{\alpha}]\cap (d^{1/2}\RR)\ne\emptyset)  
	=&\Pc_{d^{1/2} xn^{1/4}}(W[0,n^{1/2+\alpha}]\cap (d^{1/2} n^{1/4}\RR)\ne\emptyset) \nonumber\\
	\le&\Pd_{[ xn^{1/4}]}(\dist(S[1,n^{1/2+\alpha}],R_n)\le   \lambda n^{1/4} )+ \varepsilon_n(\lambda)  \nonumber
	\\
	\le&\Pd_{[xn^{1/4}]}(\dist(S[1,\infty),R_n)\le   \lambda n^{1/4})+ \varepsilon_n(\lambda), \label{ds1Rn}
	\end{align}
	
	\noindent where $\dist(A, B):= \min_{a\in A,\, b \in B} |x-y|$ and  $$\varepsilon_n(\lambda):= \p \Big( \max_{1\le k \le  n^{1/2+\alpha}} \sup_{k-1\le t \le k} | S_k- d^{-1/2} W_t|  \ge  \lambda n^{1/4} \Big). $$

	\noindent By the coupling \eqref{KMT} and the Brownian fluctuations (Lemma 1.1.1 in \cite{CR81}),  we get that $$ \varepsilon_n(\lambda) \to 0, \qquad n \to\infty.$$

	For the probability term in \eqref{ds1Rn},  if $(S_k)$ approaches $R_n$ without touching it, then we can use the strong Markov property upon the stopping time $\inf\setof{k\ge 1}{\dist(S_k,R_n)\le   \lambda n^{1/4} }$ to see that
	$$
	\Pd_{[ xn^{1/4}]}(\dist(S[1,\infty),R_n)\le   \lambda n^{1/4} ) 
	\le
	\Pd_{[ xn^{\frac 1 4}]}(\tau^+_{R_n}<\infty)
	+\sup_{\dist(y,R_n)\le   \lambda n^{1/4} }\Pd_y(\tau^+_{R_n}=\infty).$$
	
	\noindent It follows that \begin{eqnarray*}
		&& \Emodel\bracks*{\Pc_{d^{1/2} x}(\hit_{d^{1/2}\RR}<\infty)1_{\A_{\infty, K}}}
		\\
		&\le&
		o(1)+ \varepsilon_n(\lambda) + \Emodel\Big[\Pd_{[ xn^{\frac 1 4}]}(\tau^+_{R_n}<\infty)\, 1_{\A_{n,K}}\Big]
		+ \Emodel\Big[\sup_{\dist(y,R_n)\le   \lambda n^{1/4} }\Pd_y(\tau^+_{R_n}=\infty)\Big],
	\end{eqnarray*}
	
	\noindent where $o(1)+\varepsilon_n(\lambda) \to 0$ as $n \to\infty$. Applying Theorem \ref{lem:inter1} (letting $\lambda \to 0$ after letting $n\to\infty$) yields \eqref{eq:intersectionConvergence} and completes the proof of  \eqref{xi-lower}.
\end{proof}

\section{Intersection probabilities}\label{s:intersection}

This section is devoted to the proof of Theorem \ref{lem:inter1}. As already 
observed in Le Gall and Lin \cite{LeGall-Lin-lowdim} and Zhu \cite{Zhu-cbrw}, 
it will be more convenient to consider the following two models of BRWs indexed 
by infinite Galton-Watson forests $\tree_\infty$ and $\tree^*_\infty$.

To construct the first model $\tree_\infty$, consider an infinite ray 
$(\varnothing_n)_{n\ge 0}$ called spine. For each $n\ge 0$, $\varnothing_n$ 
gives birth to $k$ children with probability $\sum_{i=k+1}^\infty p_i$  (This 
is well-defined since $\sum_{i=0}^\infty ip_i=1$). To each of these children, 
we attach an independent copy of $\mathcal T$, and the resulting structure is 
denoted by $\tree_\infty$. We explore $\tree_\infty$ in lexicographical 
order (also known as depth-first search, see \Cref{fig} for an illustration) 
and denote this sequence by $(u_k)_{k\ge 0}$ with 
$u_0:=\varnothing_0$.  We view $\varnothing_0\equiv \varnothing$ as the root of $\tree_\infty$. 

The second model $\tree_\infty^*$ is based on the same structure of the spine 
$(\varnothing_n)$. Now except for $\varnothing_0$, on each $\varnothing_n (n\ge 
1)$ we employ the same construction as $\tree_\infty$, and to $\varnothing_0$ 
we attach an independent copy of $\tree$. We explore $\tree_\infty^*$ in 
lexicographical order {\it ignoring} the vertices $(\varnothing_n)_{n\ge 1}$, and 
denote the resulting sequence by $(u_k^*)_{k\ge 0}$ .

\begin{figure}[ht]
	\centering
	\includegraphics[scale=0.8]{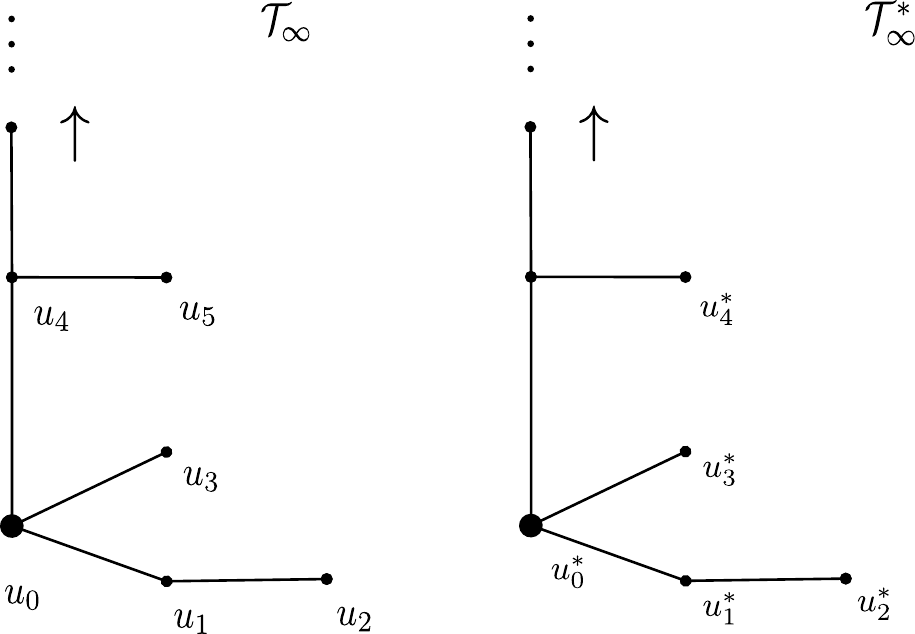}
	
	\caption{\leftskip=1.6truecm \rightskip=1.6truecm  \small A  comparison of the 
		exploration sequences on the two models, with $u_0=u^*_0=\varnothing_0\equiv \varnothing$ and $u_4=\varnothing_1$.  Notice that the two models do not turn 
		out to be the same structure shown here with same probability.}
	\label{fig}
\end{figure}

Let $V_\tree$, $V_{\tree_\infty}$ and $V_{\tree_\infty^*}$ be the BRW indexed 
by $\tree$, $\tree_\infty$ and $\tree_\infty^*$, respectively. Denote by 
$\P_x^{(\tree)}$ (resp: 
$\Pbrw_x$, $\P_x^{(\tree_\infty^*)}$) 
the law of $V_\tree$   
(resp: $V_{\tree_\infty}$, $V_{\tree_\infty^*}$) with $V_\tree(\varnothing)=x$
(resp: $V_{\tree_\infty}(\varnothing)=x$, $V_{\tree_\infty^*}(\varnothing)=x$). 
To ease the notation, we shall omit the subscripts in $V$ when it is clear from 
the content. As such, the law of $R_n$, under $\p$, is the same as that of 
$\{V(u), u \in \tree\}$ under $\P_0^{(\tree)}(\bullet \,|\, \#\tree=n)$.

Indeed, $\tree_\infty$ is intuitively half of a Galton-Watson tree conditioned 
to be infinite, and $\tree_\infty^*$ is an artificial model constructed to 
guarantee the following invariance by translation (see e.g. \cite[Section 
2]{bai2020capacity}): under $\P_0^{(\tree^*_\infty)}$, for any $i\ge 0$, 
\begin{equation}\label{invariance} (V(u^*_{i+k})- V(u^*_i), k\ge 0)  \law 
(V(u^*_k), k\ge 0). \end{equation}

Let $S$ be as before a simple random walk on $\z^d$ (independent of $V_\tree$, $V_{\tree_\infty}$ and $V_{\tree_\infty^*}$).  For any $0\le j < k$, write 
$$S[j, k]:= \{S_i: j\le i \le k\}, \,\, V[j, k]:=\{V(u_i): j\le i \le k\}, 
\,\,  V^*[j, k]:=\{V(u^*_i): j\le i \le k\},$$

\noindent with similar notations for $S[j, k)$, $V[j, k)$ and $V^*[j,k)$ (with 
possibility that $k=\infty$).

The main result of this section is the following theorem, from which we will deduce  
\eqref{eq:inter1} in Section \ref{sec:3.2}.  

\begin{theorem}\label{thm:mainfinite}
	Assume \eqref{hyp-theta} and \eqref{hyp-GW}. In dimensions $d=3,4,5$,
	\begin{align}
	&\lim_{\lambda\rightarrow 0+}\limsup_{n\rightarrow\infty}\Ebrw 
	\Big[\sup_{\dist(x,V[0,n])<\lambda n^{1/4}}\Pd_x(S[0,\infty)\cap 
	V{[0,3n/2]}=\emptyset) \Big]=0, \label{SV}\\
	&\lim_{\lambda\rightarrow 0+}\limsup_{n\rightarrow\infty}\mathbf 
	E^{(\tree_\infty^*)} 
	\Big[\sup_{\dist(x,V^*[0,n])<\lambda n^{1/4}}\Pd_x(S[0,\infty)\cap 
	V^*{[0,3n/2]}=\emptyset)\Big]=0. \label{SV*}
	\end{align}
\end{theorem}

The parameter $3/2$ can be replaced by any fixed constant $c>1$ with the same 
proof.

The proof of Theorem 
\ref{thm:mainfinite} is based on some ideas taken from Lawler \cite[proof of 
Lemma 2.5]{Lawler1996} who studied the intersection of two independent simple random 
walks. The strategy can be summarised as follows: First we show that in 
dimensions $d\in \{3, 4, 5\}$, there is a small but non-negligible probability 
that the SRW and the BRW (under $\P^{(\tree_\infty)}$ or under a certain 
conditional probability of $\P^{(\tree)}$) intersect (\Cref{lem:existproba} and 
\Cref{cor:4.4}). The next step is to use the optional lines for the 
BRW $V_{\tree_\infty}$ to create enough independence when we cut the BRW into 
small pieces. This will be done in \Cref{lem:4.5} which describes the law of 
the BRW $V_{\tree_\infty}$ between two random times, then we can iterate these 
random 
times and prove a certain rate of decay in the non-intersection probability 
between $S$ and $V_{\tree_\infty}$ (\Cref{lem:epsilonMd2}). As will be 
shown in \Cref{lem:dist_spine}, 
we may compare $ \tree^*_\infty $ and $ \tree_\infty $, then deduce the 
corresponding result for $V_{\tree^*_\infty}$ (Corollary \ref{cor:epsilonMd2-new}). 
This together with the  stationary increments \eqref{invariance}  of  
$V_{\tree^*_\infty}$  
imply  
\eqref{SV*} (Section \ref{s:iteration}). Finally the same comparison argument 
between $\tree_\infty$ and $\tree^*_\infty$ yields \eqref{SV}.

\subsection{Some preliminary estimates on a Galton-Watson forest}

At first we recall some facts on the coding of a Galton-Watson forest,  see  Le 
Gall \cite[Chapter 1, page 254]{LeGall2005}.  Let $(H_k)_{k\ge 0}$ be the 
height process obtained from a sequence of i.i.d. copies of $\tree$,  by 
concatenating their height functions (so the root of each copy of $\tree$ has 
height $0$).  Let $(L_k)_{k\ge 0}$ be the associated Lukasiewicz walk, which is a random 
walk on $\z$ starting from $L_0=0$ and  with jump distribution $\p (L_1=i)= 
p_{i+1}$ for $i\ge -1$ (where $\p$ denotes the probability which governs this 
sequence of i.i.d. copies of $\tree$), coupled with $(H_k)$ such that
\begin{equation}\label{HkL} 
H_k{=} \sum_{i=0}^{k-1} 1_{\{L_i= \min_{i\le j \le k} 
	L_j\}},\qquad k\ge 0.
\end{equation}


Now we describe the height process $(|u^*_n|)_{n\ge 0}$ of the vertices of 
$\tree_\infty^* \backslash\{\varnothing_k, k\ge 1\}$, where for any $u\in 
\tree_\infty^*$, we denote by $|u|$ its graph  distance between $u$ and the 
root $\varnothing_0$. The main difference between $(|u^*_n|)_{n\ge 0}$ and the 
aforementioned height process $(H_k)$ lies at the spine $(\varnothing_k)_{k\ge 
	0}$, because each $\varnothing_k$ with $k\ge1$, has an  
offspring distribution different from $(p_i)$, and the height of the spine is 
defined as $|\varnothing_k|=k$. 

For each $k\ge 1$, denote by $D_k$  the number of children of $\varnothing_k$, 
then $D_k, k\ge 1$ are i.i.d. with distribution 
$$\mathbf P^{(\tree_\infty^*)}(D_k= j)=\sum_{i=j+1}^\infty 
p_i, \qquad j\ge 0.$$

\noindent Since $\varnothing_0$ is different from other $\varnothing_k$ in $\mathcal T^*_\infty$, we denote $D_0:=1$.  We can view $\mathcal T^*_\infty$ as the spine together with $D_k$ i.i.d. 
copies of $\mathcal T$ attached to $\varnothing_k$.  Define $\Sigma_0:=1$ and 
\begin{equation} \label{Sigma_k}\Sigma_k:= 1+ 
\sum_{i=1}^k D_i, \qquad k\ge 1. 
\end{equation}
Then $(\Sigma_k)$ counts the number of copies of $\tree$ attached to the spine 
until $\varnothing_k$.

Let $(H_n)_{n\ge0}$ be the height process 
of this sequence of i.i.d. copies of $\tree$, and $(L_n)$ the 
associated Lukasiewicz walk (under the probability $\mathbf 
P^{(\tree_\infty^*)}$ in lieu of $\p$).
Moreover, if we denote by $T^L_{-j}:= \inf\{n\ge 0: L_n =-j\}$ for any $j\ge 
1$, then $T^L_{-\Sigma_k}- T^L_{-\Sigma_{k-1}}$ is exactly the total progeny of 
those $D_k$-trees attached to $\varnothing_k$. 
It follows that for any $n\ge 1$, $$  |u^*_n| = \begin{cases} H_n , \qquad & 
\mbox{if $n < T^L_{-1}$}, \\
H_n +k +1, \qquad & \mbox{if $T^L_{-\Sigma_{k-1}}\le n < T^L_{-\Sigma_k}$ for some 
	$k\ge1$}.
\end{cases} $$

If we define for any $n\ge 1$, \begin{equation}\label{sigma_n}
\sigma_n:= \min\{k\ge 0: \Sigma_k > -\min_{0\le i \le n} L_i\},\end{equation}

\noindent then \begin{equation}\label{H+sigma}|u^*_n|= H_n + \sigma_n + 
1_{\{\sigma_n>0\}}, \qquad \forall n\ge 1. \end{equation}

Write $v_n(s), 0\le s\le 1$, the linear interpolation of $n^{-1/4} 
V(u^*_{\lfloor n s\rfloor})$: $$ v_n(s):= n^{-1/4} \Big(V(u^*_{\lfloor n 
	s\rfloor}) + (ns- \lfloor n s\rfloor) ( V(u^*_{\lfloor n s\rfloor}+1)- 
V(u^*_{\lfloor n s\rfloor}))\Big).$$

The following result  describes the growth of 
$v_n$ as well as the increments of its positions listed  in lexicographic 
order: 

\begin{lemma}\label{l:modulo} Assume \eqref{hyp-theta} and \eqref{hyp-GW}. For 
	any $0< b < \frac{q}{4} -1$, there is some positive constant $C_{b, q}$ such 
	that $$ \mathbf E^{(\tree_\infty^*)}\Big[ \sup_{0\le s \neq t \le 1} 
	\frac{|v_n(t)- v_n(s)|^q}{|t-s|^b}\Big] \le C_{b, q}.$$
\end{lemma}

{\noindent\it Proof.} By the Garsia-Rodemich-Rumsey lemma (see 
\cite[(3.b)]{MR680660}), it suffices to show that 
for all $0\le s\le t \le 1$ and $n\ge1$, 
\begin{equation}\label{v-increment} 
\mathbf E^{(\tree_\infty^*)} [|v_n(t)- v_n(s)|^q]\le C_q\, (t-s)^{q/4}. 
\end{equation}

\noindent This is equivalent to show that for any 
$0\le j < k \le n$, $$ \mathbf E^{(\tree_\infty^*)} [| V(u_k^*)- V(u_j^*)|^q] 
\le C_q \, (k-j)^{q/4}.$$

By the translation invariance \eqref{invariance}, it is enough to show 
that for any $k\ge 1$, $$ \mathbf E^{(\tree_\infty^*)}[| V(u^*_k)|^q] \le C_q 
\, k^{q/4}.$$

\noindent
Note that conditionally on  $\{|u^*_k|= \ell\}$, $V(u^*_k) \law  X_\ell$ the 
sum of $\ell$ i.i.d. variables distributed as $\theta$.   By 
\eqref{petrov-moment},  $$ \mathbf E^{(\tree_\infty^*)} [| 
V(u^*_k)|^q ]\le C'_q \, 
\mathbf E^{(\tree_\infty^*)} [|u^*_k|^{q/2}].$$

Recall \eqref{H+sigma}, it suffices to show that 
\begin{eqnarray}
\mathbf E^{(\tree_\infty^*)}[ H_k ^{q/2}] \le C_q k^{q/4}, \label{Hk-moment}
\\
\mathbf E^{(\tree_\infty^*)}[ \sigma_k ^{q/2}] \le C_q k^{q/4}. 
\label{Ik-moment}
\end{eqnarray}

The estimate \eqref{Hk-moment} is known, for instance it follows from Marzouk 
\cite[(5)]{Marzouk}.  To show \eqref{Ik-moment}, we remark that 
$\mathbf E^{(\tree_\infty^*)}[D_1]< \infty$ (as $\sum_{i\ge0} i^2 p_i 
< \infty$ by \eqref{hyp-GW}). Applying the renewal theorem (Gut \cite{Gut}, Theorem 2.5.1) to the 
positive random walk $(\Sigma_k)$,   we have  that for all $k\ge 1$, $$ \mathbf E^{(\tree_\infty^*)} 
[\sigma_k^{q/2}] \le C_q\, \mathbf 
E^{(\tree_\infty^*)}\bracks*{\big|\min_{0\le i \le k} 
	L_i\big|^{q/2}}
\le 
C_q k^{q/4}, $$ 

\noindent
where the last inequality follows from Kortchemski (\cite{Kortchemski}, Proposition 8). This shows \eqref{Ik-moment} and completes the proof of the Lemma.
\hfill$\Box$

\medskip
As a consequence of  Lemma \ref{l:modulo}, we get the following 
estimate for future use: 
For any $0< \zeta< \frac14 - \frac1{q}$. There exists some positive constant 
$a$ and $C= C_{a, \zeta}$ such that  all $ n\ge1$ and $0<\varepsilon<1$, we 
have  \begin{equation}\label{V-increment} \P^{(\tree_\infty^*)}\Big( 
\max_{0\le k < \frac1\varepsilon} \max_{0\le j \le \varepsilon n} | V(u_{j+ k 
	\varepsilon n}^*)- V(u_{k \varepsilon n}^*)| \ge \varepsilon^\zeta n^{1/4}\Big) 
\le  C  \, \varepsilon^a.
\end{equation}

\noindent In fact, let $\zeta q < b < \frac{q}4-1$. Observe that the probability term in \eqref{V-increment} is less than $ 
\P^{(\tree_\infty^*)}(\sup_{0\le s < t \le 1, t-s \le \varepsilon} |v_n(t)- 
v_n(s)| \ge \varepsilon^\zeta)\le \P^{(\tree_\infty^*)}(\sup_{0\le s < t \le 1} 
\frac{|v_n(t)- v_n(s)|^q}{(t-s)^b}  \ge \varepsilon^{\zeta q  - b})$, therefore 
\eqref{V-increment} follows from Lemma \ref{l:modulo} with $a:= b - \zeta q>0$.

Another consequence is that, by taking $s=0$ in Lemma \ref{l:modulo} and 
eliminating  $|t-s|^b$ term, we obtain an upper bound for the moments of the 
maximum of $V_{\tree^*_\infty}$: \begin{equation}\label{maxV>} \E^{(\tree_\infty^*)} \big[\max_{0\le i \le n} |V(u^*_i)|^q \big] \le C_q\, 
n^{q/4}, \qquad \forall\, n\ge 1.
\end{equation}

We present now the aforementioned comparison between $\tree_\infty$ and 
$\tree^*_\infty$.  Notice that if we drop the root $\varnothing_0$ and the Galton--Watson tree attached to   $\varnothing_0$ from $\tree^*_\infty$, then the remaining 
structure is distributed as $\tree_\infty$. Denote by ${\tt t}_0^*$ the 
population of the subtree rooted at $\varnothing_0$ (without counting $\varnothing_0$). Then 
$p_0=\P^{(\tree^*_\infty)}({\tt t}_0^*=0)$. 

\begin{figure}[ht]
	\centering
	\includegraphics[scale=0.8]{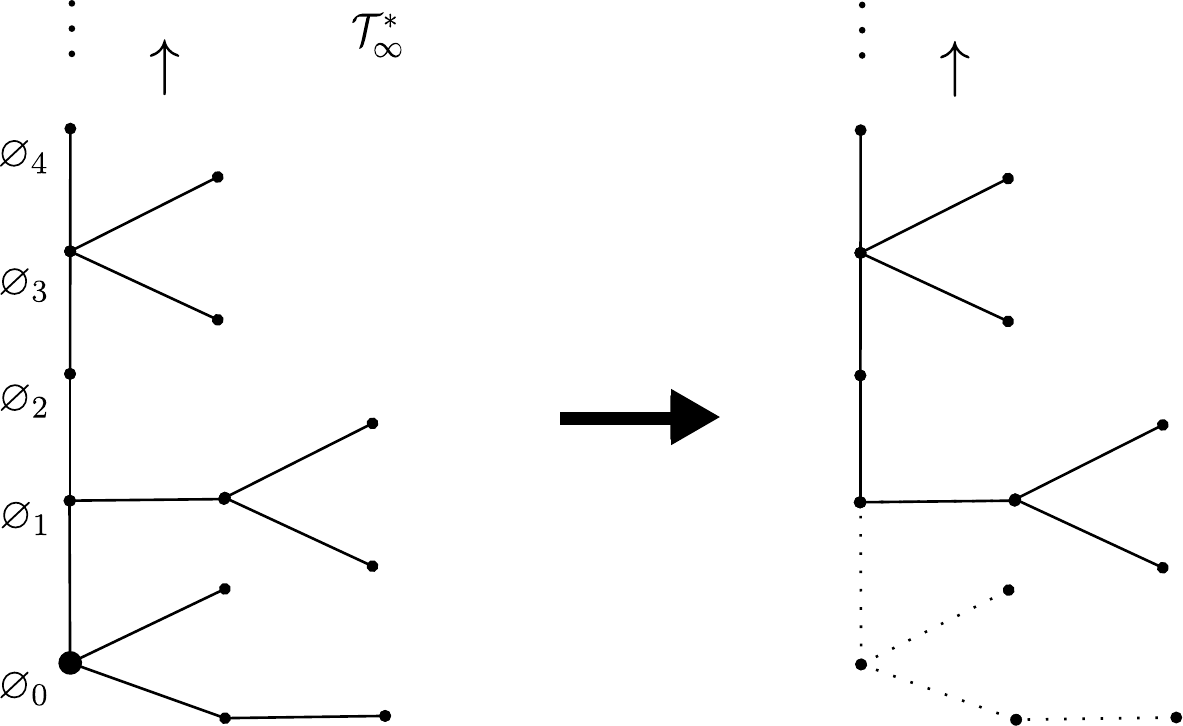}
	\caption{
		\leftskip=1.6truecm \rightskip=1.6truecm \small An illustration for 
		\Cref{lem:dist_spine}. After deleting the first subtree in 
		$\tree_\infty^*$, the remaining structure is distributed as 
		$\tree_\infty$.}
\end{figure}
\begin{lemma}\label{lem:dist_spine} Assume \eqref{hyp-theta} and 
	\eqref{hyp-GW}. 
	Under $\P^{(\tree^*_\infty)}$, we may find a subgraph of $\tree^*_\infty$ 
	distributed as $\tree_\infty$ under $\P^{(\tree_\infty)}$. Abuse the notation 
	$(V(u_i))_{i\ge 0}$ for the BRW indexed by it (translated so that it starts at 
	$0$), then
	\begin{equation}\label{coupling2trees1} 
	\big(V(u_i), u_i \not\in (\varnothing_j)_{j=0}^\infty\big)_{1\le i \le  n} \, 
	\subset \, \big(V(u^*_{i+{\tt t}_0^*}) - V(\varnothing_1)\big)_{1\le i \le n},  
	\mbox{ almost surely}, \end{equation}
	\noindent and for any $\varepsilon>0$,   the following event happens with probability $1-o(1)$: 
	\begin{equation}\label{coupling2trees}   \big(V(u_i)\big)_{1\le i \le n} \, \supset\, \big(V(u^*_{i+{\tt t}_0^*}) - V(\varnothing_1)\big)_{1\le i \le (1-\varepsilon) n}, 
	\end{equation}
	\noindent where $o(1)\to0$ as $n\to\infty$.  

Under this construction, both $\big(V(u_i)\big)_{1\le i \le n}$ and $\big(V(u^*_{i+{\tt t}_0^*}) - V(\varnothing_1)\big)_{1\le i \le n}$ are independent of $V(\varnothing_1)$, thus we may add $V(\varnothing_1)$ on both sides and deduce that, there is a coupling between two tree models and a random variable $X\sim\theta$, so that
\begin{equation}\label{eq:independentX1}
	\big(V(u_i), u_i \not\in (\varnothing_j)_{j=0}^\infty\big)_{1\le i \le  n}+X \, 
	\subset \, \big(V(u^*_{i+{\tt t}_0^*})\big)_{1\le i \le n},  
	\mbox{ almost surely}, 
\end{equation}
\begin{equation}\label{eq:independentX2}
\big(V(u_i)\big)_{1\le i \le n}+X \, \supset\, \big(V(u^*_{i+{\tt t}_0^*}) \big)_{1\le i \le (1-\varepsilon) n}, \text{ with probability }1-o(1),
\end{equation}
and $X$ is independent of $\big(V(u_i)\big)_{1\le i\le n}$.
 
 Moreover, under $\mathbf P^{(\tree_\infty)}$ or $\mathbf 
	P^{(\tree_\infty^*)}$, the following happens with probability $1-o(1)$ as 
	$n\rightarrow\infty$:
	\begin{equation}\label{eq:dist_spine0}
	\max_{0\le i\le n} 
	\dist\big(V(\varnothing_i),\setof*{V(u)}{u\text{ is on one of the subtrees rooted 
			at }\varnothing_0,\cdots,\varnothing_n}\big)
	\le n^{\frac 1 q+\varepsilon},
	\end{equation}
	where $q>4$ is given in \eqref{hyp-theta}.
\end{lemma}

We may replace $(1-\varepsilon) n$ by $n- n^{\frac12 +\varepsilon}$ in \eqref{coupling2trees}.

\begin{proof} Given $\tree_\infty^*$, if we denote the depth-first sequence starting at $\varnothing_1$ (including the spine) by $\widetilde V[0,\infty)$,  then up to a shift,  it is identically distributed as $V[0,\infty)$ under $\tree_\infty$. In other words,  under $\P^{(\tree^*_\infty)}$, $\widetilde V[0,\infty) - V(\varnothing_1)$ is distributed as $\P^{(\tree_\infty)}$ (and independent of $V(\varnothing_1)$). We take it as a version of $(V(u_i))_{i\ge 0}$. Then \eqref{coupling2trees1} follows.

	Recall \eqref{sigma_n}. Observe that  $\varnothing_{\sigma_n}$ is the last spine vertex at which one of the rooted subtrees intersects with  $V^*[0,n]$. Then $\P^{(\tree^*_\infty)}$-a.s., for any $k\ge 1$, \begin{equation}\label{couplingsigman}\big(V(u_i)\big)_{1\le i \le k+\sigma_n} \supset\big(V(u^*_{i+{\tt t}_0^*}) - V(\varnothing_1)\big)_{1\le i \le k} . \end{equation}
	
	\noindent By \eqref{Ik-moment}, $\P^{(\tree^*_\infty)}(\sigma_n > \varepsilon n) \to 0$ as $n \to \infty$. This implies   \eqref{coupling2trees}.  
	We mention that with $\sigma_n$, we may re-write \eqref{coupling2trees1} as 
	\begin{equation}\label{coupling2trees-sigma} \big(V(u_i), u_i \not\in 
	(\varnothing_j)_{0\le j\le\sigma_n}\big)_{1\le i \le  n} \, \subset \, 
	\big(V(u^*_{i+{\tt t}_0^*}) - V(\varnothing_1)\big)_{1\le i \le n} , \qquad 
	\mbox{$\P^{(\tree^*_\infty)}$-a.s.} \end{equation}   

    For \eqref{eq:independentX1} and \eqref{eq:independentX2}, it suffices to take $X=V(\varnothing_1)$ on the right hand side of \eqref{coupling2trees1} and \eqref{coupling2trees}, then add it to both sides.

	It remains to show \eqref{eq:dist_spine0}. Let $U_n:=\braces*{\text{vertices on the subtrees rooted at }\varnothing_0,\cdots,\varnothing_n}$. We claim that there exists $C>0$ such that
	\begin{equation}\label{eq:dist_spine}
	\max_{0\le i\le n}\dist(\varnothing_i,U_n)<C\log n \text{ with probability }1-o(1),
	\end{equation}
	
	\noindent where we abuse the notation $\dist(\cdot,\cdot)$ both for graph-distance between vertices and Euclidean distance between points in $\mathbb Z^d$. Indeed, $\max_{0\le i\le n}\dist(\varnothing_i,U_n)\ge C \log n$ means that  there are $C \log n$ consecutive vertices on the spine that give no  offspring at all, which happens with probability at most $n(1-p_0)^{C\log  n}=o(1)$ by taking $C$ large enough.
	
	Given the condition \eqref{eq:dist_spine}, for each $\varnothing_i$, we can 
	find a point in $U_n$ at most $C\log n$ away on the tree, thus
	by union bounds, the probability of \eqref{eq:dist_spine0} is at most
	\[
	n\, \p(X_{C\log n}\ge n^{\frac 1 q+\varepsilon})+o(1),
	\]
	\noindent where as before, under $\p$, $(X_k)_{k\ge 0}$ is a random walk on $\z^d$ with step 
	distribution $\theta$.  By \eqref{petrov-moment},  the conclusion  follows from  Chebyshev's inequality.
\end{proof}

Using the coupling between $\tree_\infty$ and $\tree^*_\infty$ in Lemma \ref{lem:dist_spine}, we get two useful estimates for the BRW under $\P^{(\tree_\infty)}$. First, let $q$ be as in \eqref{hyp-theta}. We claim that
\begin{equation}\label{maxV>2} 
\E^{(\tree_\infty)} \big[\max_{0\le 
	i \le n} |V(u_i)|^q \big] \le C'_q \, n^{q/4}, \qquad \forall\, n\ge 1.
\end{equation}

\noindent In fact, we deduce from \eqref{coupling2trees-sigma} that  under $\P^{(\tree^*_\infty)}(\bullet \, |\, {\tt t}_0^*=0)$, $$ \max_{1\le i\le n-1} |V(u_i)|  \le \max\Big(\max_{1\le i \le n} | V(u^*_i)- V(\varnothing_1)|,  \max_{1\le i \le \sigma_n} | V(\varnothing_i)-V(\varnothing_1)| \Big). $$ 

\noindent 
By \eqref{petrov-moment}  and \eqref{Ik-moment},
$$\E^{(\tree^*_\infty)}\big(\max_{1\le i \le \sigma_n} | V(\varnothing_i)|^q\big) = \e \big(\max_{1\le i \le \sigma_n} |X_i|^q\big) \le C\, \E^{(\tree^*_\infty)}\big(\sigma_n^{q/2}\big) \le C' n^{q/4}.$$

\noindent Since $V(\varnothing_1)$ is distributed as $\theta$ then has finite $q$-th moment, we easily deduce \eqref{maxV>2} from \eqref{maxV>} with $C'_q$ depending on $C_q, C'$ and $p_0=\P^{(\tree^*_\infty)}({\tt t}_0^*=0)$.

Another estimate concerns the increment of $V$: For any $\varepsilon>0$, we have \begin{equation} \label{max-V-increment1} \P^{(\tree_\infty)}\Big( \max_{0\le i\le n} |V(u_{i+1})-V(u_i)| > \varepsilon n^{1/4}\Big) \to 0, \qquad n\to \infty.
\end{equation}

\noindent To show \eqref{max-V-increment1}, we work again under $\P^{(\tree^*_\infty)}(\bullet \, |\, {\tt t}_0^*=0)$. Note that 
$$\max_{0\le i\le n}\dist(u_i, u_{i+1}) \le \max_{0\le i\le n}\dist(u^*_i, u^*_{i+1}).$$

Note that for any $k\ge 2$,  $   \P^{\tree^*_\infty}(\dist(u^*_0, u^*_1) 
\ge k)= p_0 \, \prod_{i=1}^{k-2}\P^{\tree^*_\infty}(D_i=0)=p_0^{k-1}. $ Then 
for any $K> (\log \frac1{p_0})^{-1}$, \begin{eqnarray*}
	\P^{\tree_\infty}(\max_{0\le i\le n}\dist(u_i, u_{i+1})\ge K \log n) 
	&\le&
	\frac1{p_0} \P^{\tree^*_\infty}(\max_{0\le i\le n}\dist(u^*_i, u^*_{i+1})\ge K \log n) 
	\\
	&\le &
	\frac{n}{p_0}  \P^{\tree^*_\infty}(\dist(u^*_0, u^*_1) \ge K \log n) 
	\\
	&\le& n \, p_0^{ K \log n -2} \to 0.
\end{eqnarray*}

\noindent It follows from the union bound and \eqref{petrov-moment} that
\begin{eqnarray*} \mbox{LHS of \eqref{max-V-increment1}} &\le&
	n p_0^{ K \log n -2} + C\, \sum_{i=0}^n (\varepsilon n^{1/4})^{-q} \E^{(\tree_\infty)}\big( \dist(u_i, u_{i+1})^{q/2} 1_{\{\dist(u_i, u_{i+1}) \le K \log n\}} \big).
	\\
	&\le&
	n p_0^{ K \log n -2} + C\, (n+1) (\varepsilon n^{1/4})^{-q}\,  (K \log n)^{q/2} 
	\\
	&\to & 0, \qquad n\to\infty,
\end{eqnarray*}
which shows \eqref{max-V-increment1} as $q>4$.

\medskip
The following Lemma describes how small can the BRW be:

\begin{lemma}\label{L:maxV} Assume \eqref{hyp-theta} and \eqref{hyp-GW}. There 
	exists some $C, C'>0$ such that for  any $ \varepsilon \in (0, 1)$, 
	\begin{equation}\label{maxV<} 
	\begin{aligned}
	&\limsup_{n\to \infty}\Pbrw\Big( \max_{0\le i \le 
		n} |V(u_i)| \le \varepsilon   n^{1/4}\Big) \le C' \, e^{- 
		C/\varepsilon},\\
	&\limsup_{n\to \infty}\mathbf P^{(\tree_\infty^*)}\Big( \max_{0\le i \le 
		n} |V(u^*_i)| \le \varepsilon   n^{1/4}\Big) \le C' \, e^{- 
		C/\varepsilon}.
	\end{aligned}
	\end{equation}
\end{lemma}

\begin{proof} 
	By \Cref{lem:dist_spine}, we only need to show the above estimate for $\tree_\infty^*$.
	
	Let $\varepsilon>0$ be small. The probability term in the LHS of 
	\eqref{maxV<} is less than $$ \mathbf P^{(\tree_\infty^*)}\Big( \max_{0\le i 
		\le n} |u_i^*| \le 
	\varepsilon  n^{1/2}\Big) + 
	\p\Big(\max_{0\le k \le \varepsilon n^{1/2}} 
	|X_k|\le \varepsilon n^{1/4}\Big).$$
	
	\noindent We estimate the above two probabilities separately. The second one is a classical estimate on the random walk:  By Chung \cite{Chung1948},  provided that the centered random walk $(X_k)$ has finite third moment (which is the case thanks to 
	\eqref{hyp-theta}),   we have  
	$$\limsup_{n\to \infty} \p\Big(\max_{0\le k \le 
		\varepsilon n^{1/2}} |X_k|\le \varepsilon n^{1/4}\Big) \le e^{- 
		C/\varepsilon}.$$

	By \eqref{H+sigma}, $\max_{0\le i \le n} |u^*_i|\ge \max_{0\le i \le n}  H_i 
	$.  Then
	\begin{align*}
	\mathbf P^{(\tree_\infty^*)}\Big( \max_{0\le i \le n} |u^*_i| \le \varepsilon  
	n^{1/2}\Big) 
	&\le  \mathbf P^{(\tree_\infty^*)}\Big( \cap_{k=0}^{1/\varepsilon -1}\{ \max_{0 
		\le i \le   \varepsilon^2 n}  H_{i+k \varepsilon^2 n}  \le \varepsilon  
	n^{1/2}\}\Big)
	\\
	&\le
	\mathbf P^{(\tree_\infty^*)}\Big( \max_{0 \le i \le  \varepsilon^2 n}  H_i  \le 
	\varepsilon  n^{1/2}\Big)^{1/\varepsilon -1},
	\end{align*}
	
	\noindent where in the second inequality, we use the fact (see \cite{D-LG}, Lemma 2.3.5) that for any $0\le k < 1/\varepsilon$, conditionally on $(H_i)_{0\le i<k\varepsilon^2 n}$, $(H_{i+k \varepsilon^2 n})_{i\ge 0}$  is stochastically larger than (an independent copy of)
	$(H_i)_{i\ge 0}$.  
	
	Note that  $\frac{\max_{0\le i \le j}H_i}{\sqrt{j}}\,  \wcv \, \frac{2}{\sigma_p} \, \sup_{0\le s\le 1} |\beta_s|$ (see Theorem 1.8 in \cite{LeGall2005}),  where $(\beta_t)$ stands for a standard one-dimensional Brownian motion.   There exist 
	some   (small) $C>0$ and $j_0\ge 1$ such that for all $j\ge j_0$, $  \mathbf P^{(\tree_\infty^*)}(\max_{0\le i \le j}H_i \le j^{1/2}) 
	\le e^{- C}$, it follows that for all large $n$,  $$\mathbf P^{(\tree_\infty^*)}\Big( \max_{0\le i 
		\le n} |u^*_i| \le \varepsilon  n^{1/2}\Big) 
	\le e^{ C- C/\varepsilon},$$
	
	\noindent completing the proof. 
\end{proof}

Moreover, we have the following estimate for the Green  function:
\begin{lemma}\label{GdVn} Assume \eqref{hyp-theta} and \eqref{hyp-GW}. Under 
	$\Pbrw$ or $\mathbf P^{(\tree_\infty^*)}$, for any $\varepsilon>0$, there 
	exists some $C=C_\varepsilon>0$ such that with probability at least 
	$1-\varepsilon$, $$ \sum_{x, y \in V[0, n]} \greend(x,y) \le \begin{cases}
	C\, n^{3/4}, \qquad &\mbox{if $d=3$}, \\
	C\, n^{(10-d)/4}, \qquad &\mbox{if $d=4, 5$}
	\end{cases}.
	$$
\end{lemma}
\begin{proof}
	It follows from the same argument in the proof of \cite[Lemma 4.4]{BaiHu}, by 
	replacing the factor $n^\varepsilon$ there by some large constant 
	$C=C_\varepsilon$. We omit the details. 
\end{proof}

We end this section by an absolute continuity lemma between $\P^{(\tree)}$ and 
$\Pbrw$ (see Zhu \cite{Zhu-cbrw}, (5.4) and (5.5)): Let $0\le k< n$. For any 
nonnegative measurable function $F$, we have   \begin{equation} 
\label{absolutecontinuity1} \E^{(\tree)}\Big( F(V(u_i), 0\le i \le k) \, 
\big|\,   \#\tree= n\Big)=
\Ebrw\Big( F(V(u_i), 0\le i \le k)   \,\Phi_{n,k}(L_k) \Big), \end{equation}

\noindent where $L$ denotes the Lukasiewicz walk associated to the sequence of i.i.d. copies of $\tree$ in $\tree_\infty$ and $\Phi_{n,k}(\ell):= \frac{n \P^{(\tree_\infty)}(L_{n- k}= - (\ell +1))}{(n- k) \P^{(\tree_\infty)}(L_n=-1)}$. Using the local central limit theorem for $L$ (see \cite{IV-book}, Theorem 4.2.1),  we get that for any fixed $a \in (0,1)$, there exists some $C=C_a>0$ such that for all $0\le k < a n$, $\Phi_{n,k}(\ell) \le C$ for all $\ell\ge0$, and therefore \begin{equation}\label{absolutecontinuity} \E^{(\tree)}\Big( F(V(u_i), 0\le i \le k) \, \big|\,   \#\tree= n\Big) \le 
C\, \Ebrw\Big( F(V(u_i), 0\le i \le k)  \Big). \end{equation}

As an application of \eqref{absolutecontinuity}, we deduce from \eqref{maxV>2} that there exists some $C'>0$ such that \begin{equation}\label{maxV>2bis} 
\E^{(\tree)} \big[\max_{0\le 
	i \le {n-1}} |V(u_i)|^q \, \big|\,   \#\tree= n\big] \le C'\, n^{q/4}, \qquad \forall\, n\ge 1,
\end{equation}

\noindent where $q>4$ is as in \eqref{hyp-theta}. In fact, we use 
\eqref{absolutecontinuity} and \eqref{maxV>2} to see that for the first $2n/3$ 
vertices,
$$\E^{(\tree)} 
\big[\max_{0\le 
	i \le 2n/3} |V(u_i)|^q \, \big|\,   \#\tree= n\big] \le C \, \Ebrw 
\big[\max_{0\le 
	i \le 2n/3} |V(u_i)|^q \big] \le C \, C'_q  n^{q/4},$$
\noindent  Moreover, we may reverse the order of children for each vertex in $\tree$ and 
obtain the same estimate for the first $\frac{2n}{3}$ vertices in the reversed 
tree. The two sets of vertices will cover the whole tree, unless there are at 
least $\frac n 3$ generations in the conditioned tree, which happens with probability 
$O(e^{-n^\delta})$ for some $\delta>0$ (see \cite{Kortchemski}, Theorem 2). It 
is not hard to see that under this event the expectation of $\max_{0\le i \le 
	{n-1}} |V(u_i)|^q $ converges to $0$.  Then we obtain \eqref{maxV>2bis}.

\subsection{A first bound for intersection probabilities under \texorpdfstring{$\Pbrw$}{}}

In this subsection, we focus on the first model $\tree_\infty$, since 
it fits  better  with hitting times.  Let $$\ball(n)= \{x\in \z^d: |x|\le n\}$$ 
be the ball centered at $0$ and with radius $n$ in $\z^d$.
We denote by  $\varrho_n$ the first time that the BRW $V$ (under 
$\P^{(\tree_\infty)}$ or $\P^{(\tree)}$), 
in lexicographic order,  exits from $\ball(n)$:  
\begin{equation}\label{def-rho}
\varrho_n:=\min\{i\ge 0: |V(u_i)|\ge n\}.\end{equation}

{The following result says that in dimensions $d=3, 4, 5$, the SRW $S$ and 
	the BRW $V$ 
	intersect at least with a non-negligible probability, as soon as their starting 
	points are not too far away from each other.}

\begin{lemma}\label{lem:existproba} Assume \eqref{hyp-theta} and \eqref{hyp-GW}.
	In dimensions $d=3,4,5$, for any $\varepsilon>0$ and $\kappa>0$, there exists some $\delta=\delta(\varepsilon, \kappa)>0$ such that for all large $n$,
	\[
	\Pbrw\bracks*{\inf_{|x|\le n}\Pd_x(S[0,\infty)\cap V{[0, \min(\varrho_n, \kappa n^4))}\ne\emptyset)<\delta}<\varepsilon.
	\]
\end{lemma}

Obviously  the above inequality remains true if we replace $\min(\varrho_n, 
\kappa n^4)$ by $\varrho_n$. {The truncation of $\varrho_n$ with $\kappa n^4$ 
	allows us to obtain a corresponding result (Corollary \ref{cor:4.4}) for 
	$\P^{(\tree)}(\bullet \,|\, \varrho_n < \infty)$.}
Moreover, the random walk $S$ can be replaced by any random walk with 
symmetric, bounded and irreducible jump distribution.

\begin{proof} {It is enough to prove the result for $d=5$,
		i.e. for any random walk $S$ with symmetric, bounded and irreducible jump 
		distribution on $\z^5$,
		from which we deduce the result for dimensions $d=3,4$ by applying the 
		projection to
		the random walk $S$ and the BRW $V$ simultaneously}: 
	$$
	(x_1,x_2,x_3,x_4,x_5)\mapsto
	\left\{
	\begin{array}{ll}
	(x_1,x_2,{x_3+x_4+x_5}), & d=3\\
	(x_1,x_2,x_3,{x_4+x_5}), & d=4
	\end{array}
	\right..
	$$

	Let $d=5$. For notional brevity we only deal with the case that $S$ is a simple 
	random walk on $\z^d$. The general case follows from the same arguments. 
	
	Let $\varepsilon>0$ and write $\widehat \varrho_n:= \min(\varrho_n, \kappa\, 
	n^4)$.  By \eqref{maxV>2},  there exists $0< c=c(\varepsilon) < \kappa$ 
	small enough such that
	\begin{equation}\label{eq:zetabound}
	\Pbrw( \widehat \varrho_n > c n^4)>1-\frac \varepsilon 3.
	\end{equation}
	Therefore, with probability at least $1-\frac \varepsilon 3$, we have
	\begin{equation}\label{eq:Rbound}
	V[0,c n^4]\subseteq V[0,\widehat \varrho_n)\subseteq \ball(n).
	\end{equation}
	
	Moreover, for any $k\ge 1$, by \cite[Lemma 2.11]{bai2020capacity},
	\begin{equation}\label{eq:cap_inequality_c1}
	\capd (V[0,c n^4])\ge\frac{\#V[0,c n^4]}{k+1}-\frac{\sum_{x,y\in V[0,c n^4]}\greend(x,y)}{k^2}.
	\end{equation}
	For the first term, by the law of large numbers for the cardinality of the range of BRW (Le Gall and Lin \cite{LeGall-Lin-range}),  with probability at least $1-\frac\varepsilon3$, we have (recall $d=5)$
	\[
	\#V[0,c n^4]\ge c^2n^4,
	\]
	
	\noindent where $c$ was supposed to be small enough. For the second term, by Lemma \ref{GdVn}   with probability $1-\frac\varepsilon3$ we have (again we choose a smaller $c$ if necessary), 
	\[
	\sum_{x,y\in V[0,c n^4]} \greend(x,y)\le cn^{5}.
	\]
	Combining the two estimates above and taking $k=\frac{2n}{c}$ in \eqref{eq:cap_inequality_c1}, we get that  $$ \Pbrw\big(\capd (V[0,c n^4])\ge  \frac{c^3 n^3}{5}\big) \ge 1-\frac{2\varepsilon}3, $$  and then
	\begin{equation}\label{eq:capRbound}
	\Pbrw\Big(\capd (V[0, \widehat \varrho_n))  \ge \frac{c^3 n^3}{5}\Big) \ge 1-\varepsilon.
	\end{equation}

	Now for any $|x|\le n$,
	by  \eqref{eq:discretenu} we get that 
	\begin{align*}
	\Pd_x(\tau_{V[0,\widehat \varrho_n)}<\infty)
	&=
	\sum_{y\in V[0,\widehat \varrho_n)}\greend(x,y)\Pd_y(\tau^+_{V[0,\widehat \varrho_n)}=\infty)\\
	&\ge
	c' n^{-3}\sum_{y\in V[0,\widehat \varrho_n)}\Pd_y(\tau^+_{V[0,\varrho_n)}=\infty)\\
	&=
	c' \, n^{-3}\capd({V[0,\widehat \varrho_n)}),
	\end{align*}
	
	\noindent where $c'>0$ is a constant such that $\inf_{|x-y|\le 2n}\greend(x,y)\ge c'n^{-3}$ (in dimension $5$). 
	Then by \eqref{eq:capRbound}, with probability at least $1- \varepsilon$, we have
	\[
	\inf_{|x| \le n}\Pd_x(\tau_{V[0,\widehat \varrho_n)}<\infty)\ge 
	\frac{c'c^3}{5},
	\]
	and the conclusion follows by taking $\delta=\frac{c'c^3}{5}$.
\end{proof}

\begin{remark}
	Although the result in $d=5$ implies that of $d=3,4$ by projection, the same proof does not directly work  for dimension $4$ (or $3$), due to the $\log n$ factor in the asymptotic of $\#(V[0,n])$ (see \eqref{LG-Lin}). \hfill$\Box$
\end{remark}

\medskip

{Analogous to $V_{\tree_\infty}$, the BRW $V$ under $\P^{(\tree)}(\bullet |\, 
	\varrho_n<\infty)$ intersects with $S$ with a non-negligible probability:}

\begin{corollary}\label{cor:4.4}  Assume \eqref{hyp-theta} and \eqref{hyp-GW}.
	In dimensions $d=3,4,5$, for any $\varepsilon>0$ and $\kappa>0$, there exists some $\delta=\delta(\varepsilon, \kappa)>0$ such that for $n$ large enough,
	$$
	\P^{(\tree)} \Big(\inf_{|x|\le n} \Pd_x(S[0,\infty)\cap V[0,\varrho_n)\ne\emptyset)<\delta   \, \big|\,    \varrho_n < \infty\Big)  <\varepsilon.
	$$
\end{corollary}

{\noindent\it Proof.}   Let $K>1$.  By \Cref{lem:existproba},   for any (small) $\kappa>0$, there exists some $\delta=\delta(\kappa, K, \varepsilon)>0$ such that 
\begin{equation}\label{deltakappaK}
\Pbrw(\eta_n <\delta)<\frac\varepsilon{2 K }, 
\end{equation}

\noindent where $\eta_n:=\inf_{|x|\le n}\Pd_x(S[0,\infty)\cap 
V{[0,\min(\varrho_n ,  \kappa n^4))}\ne\emptyset)$ (the values of $\kappa, K$ 
will be chosen later).  

The exact tail behavior of $\P^{(\tree)}(\varrho_n < \infty)$, when $d=1$, was obtained in Lalley and Shao (\cite{Lalley-Shao}) under the finite $3$-th moment of $(p_i)$ and \eqref{hyp-theta}. Under the finite second moment assumption \eqref{hyp-GW}, we may get a rough lower bound of $\P^{(\tree)}(\varrho_n < \infty)$ as follows. Recall that $\P^{(\tree)}(\max_{u\in \tree} |u| \ge n^2) \sim \frac{2}{\sigma^2_p n^2}$ as $n\to \infty$. If $u \in \tree$ is such that $|u|\ge n^2$, then  the probability of $\{|V(u)|> n\}$ is larger than  $\p(|X_{n^2}| > n) \ge c>0$ for some positive constant $c$. Therefore there exists some $C>1$ such that for all large $n$, $$ \P^{(\tree)}(\varrho_n < \infty) = \P^{(\tree)}( \max_{u\in \tree} |V(u)| \ge n) \ge
\P^{(\tree)}(\max_{u\in \tree} |u| \ge n^2)\, \p(|X_{n^2}| > n) \ge  \frac1{C n^2}.$$

\noindent Let $s_0=s_0(\varepsilon, C)>0$ be a small constant whose value will be determined below. We have 
\begin{eqnarray*}
	&&
	\P^{(\tree)}\Big(\max_{u\in \tree} |V(u)| \ge n, \, \#\tree \le s_0 n^4\Big)
	\\
	&=& \sum_{j=1}^{s_0 n^4} \P^{(\tree)}\Big(\max_{u\in \tree} |V(u)| \ge n \, \big|\, \#\tree=j\Big) \, \P^{(\tree)}( \#\tree=j)
	\\
	&\le&  C''\, \sum_{j=1}^{s_0 n^4} \, \frac{j}{n^4}\, j^{-3/2}
	\\
	& \le&   2 C'' s_0^{\frac12}\, n^{-2},
\end{eqnarray*}

\noindent where the first inequality follows from \eqref{maxV>2bis} and the 
fact that  $\P^{(\tree)}( \#\tree=j) \sim C''' j^{-3/2}$ as $j \to \infty$. 
Fix $s_0>0$ small enough such that $2 C'' s_0^{\frac12} < 
\frac{\varepsilon}{2C}$, we get that \begin{equation}\label{Tsmall} 
\P^{(\tree)}(\#\tree \le s_0 n^4 \, | \, \varrho_n < \infty)  \le 
\frac\varepsilon2.
\end{equation}

Now we choose $\kappa:=\frac{s_0}{2}$.  Note that \begin{eqnarray*} \P^{(\tree)}(\eta_n < \delta,  \, \#\tree > s_0 n^4 \, | \, \varrho_n < \infty)  
	&\le & C n^2\,  \P^{(\tree)}(\eta_n < \delta, \,  \#\tree > s_0 n^4 )
	\\
	&=&
	C n^2\,   \sum_{j>s_0 n^4} \P^{(\tree)}(\eta_n < \delta \, |\,   \#\tree= j) \, \P^{(\tree)}(\#\tree= j ).
\end{eqnarray*}

\noindent By \eqref{absolutecontinuity},  $  \P^{(\tree)}(\eta_n < \delta \, |\,   \#\tree= j) \le C'\, \Pbrw(\eta_n < \delta)$ for all $j\ge s_0 n^4$. It follows that    \begin{eqnarray*} \P^{(\tree)}(\eta_n < \delta,  \, \#\tree > s_0 n^4 \, | \, \varrho_n < \infty)  
	&\le & 
	C n^2\,  C' \, \Pbrw(\eta_n < \delta)\,  \P^{(\tree)}(\#\tree > s_0 n^4)
	\\
	&\le&
	K\, \Pbrw(\eta_n < \delta),
\end{eqnarray*}

\noindent for some numerical constant $K=K(C, C', s_0)>0$. With such choice of $K$,  \eqref{deltakappaK} says that $\P^{(\tree)}(\eta_n < \delta,  \, \#\tree > s_0 n^4 \, | \, \varrho_n < \infty)  < \frac\varepsilon2$. This in view of \eqref{Tsmall} imply that $\P^{(\tree)}(\eta_n < \delta \, |\, \varrho_n < \infty) < \varepsilon$, and then the Corollary. \hfill$\Box$

\subsection{An optional line construction under \texorpdfstring{$\Pbrw$}{}}\label{s:optionalline}

In order to explore the Markov property of the BRW,  we shall use the notion of optional line, which is a generalization of stopping times for trees. 
Let $$\ell_n:=\Big\{u \in \tree_\infty:  |V(u)|\ge n, \, \max_{v\in [\varnothing, u)} |V(v)| < n\Big\},$$

\noindent where $[\varnothing, u)$ denotes the simple path relating $\varnothing$ to $u$ (and $u$ being excluded). In other words, $\ell_n$ stands for the set of all vertices $u$ such that $|V(u)|\ge n$ and the path from the root to $u$ is contained in the ball $\ball(n)$. Note that the lexicographical order for vertices of the BRW naturally induces an order on $\ell_n$. It is immediate that under $\Pbrw$, the last vertex in $\ell_n$ is 
on the spine, and $\ell_n$ is almost surely finite and not empty.

Denote by ${\mathscr F}_{\ell_n}:=\sigma\{V(u), u:  u \nsucc \ell_n\}$, where by $u \nsucc \ell_n$, we mean that  $u$ is not a descendant of any vertex of $\ell_n$.  
Whether a particular vertex $u$ belongs to $\ell_n$ is determined by the path from the root to $u$, and this construction is an optional line in the sense of Jagers \cite{Jagers89}. In particular, $\ell_n$ is measurable with respect to ${\mathscr F}_{\ell_n}$.

Moreover, our infinite forest $\tree_\infty$ can be seen as a Galton-Watson tree with two types, distinguishing the spine and other vertices, then by Jagers (\cite{Jagers89} Theorem 4.14), conditioned on ${\mathscr F}_{\ell_n}$, the subtrees started at $\ell_n$ are independent from each other and their histories. In other words, we can view the BRW as a two-step process: firstly we construct a BRW killed upon escaping $\ball(n)$; denote the escaping points as $\ell_n$, and our second step is to grow independent branching walks from $\ell_n$, where all the points except for the last 
one gives a standard branching random walk indexed by an independent copy of the critical Galton-Watson tree $\tree$, and the last point gives an infinite BRW indexed by an independent copy of $\tree_\infty$.

Finally, for $1\le m<n$, we define $\varrho^{(n)}_m$ to be the first time that 
the simple path from $\varnothing$ to $u_{\varrho_n}$ hits $\partial\ball(m)$, 
in other words, if we list $[\varnothing, u_{\varrho_n}]$ as $\{u_{n_i}, 0\le i\le j\}$ such that 
$n_0:=0$, $n_j =\varrho_n$  and $u_{n_i}$ is the parent of $u_{n_{i+1}}$ for 
any $0\le i < j$,  then $$ \varrho^{(n)}_m:= \min\{i\in [0, j]: 
|V(u_{n_i})|\ge m\}.$$

\begin{figure}[ht]
	\centering
	\includegraphics[scale=0.45]{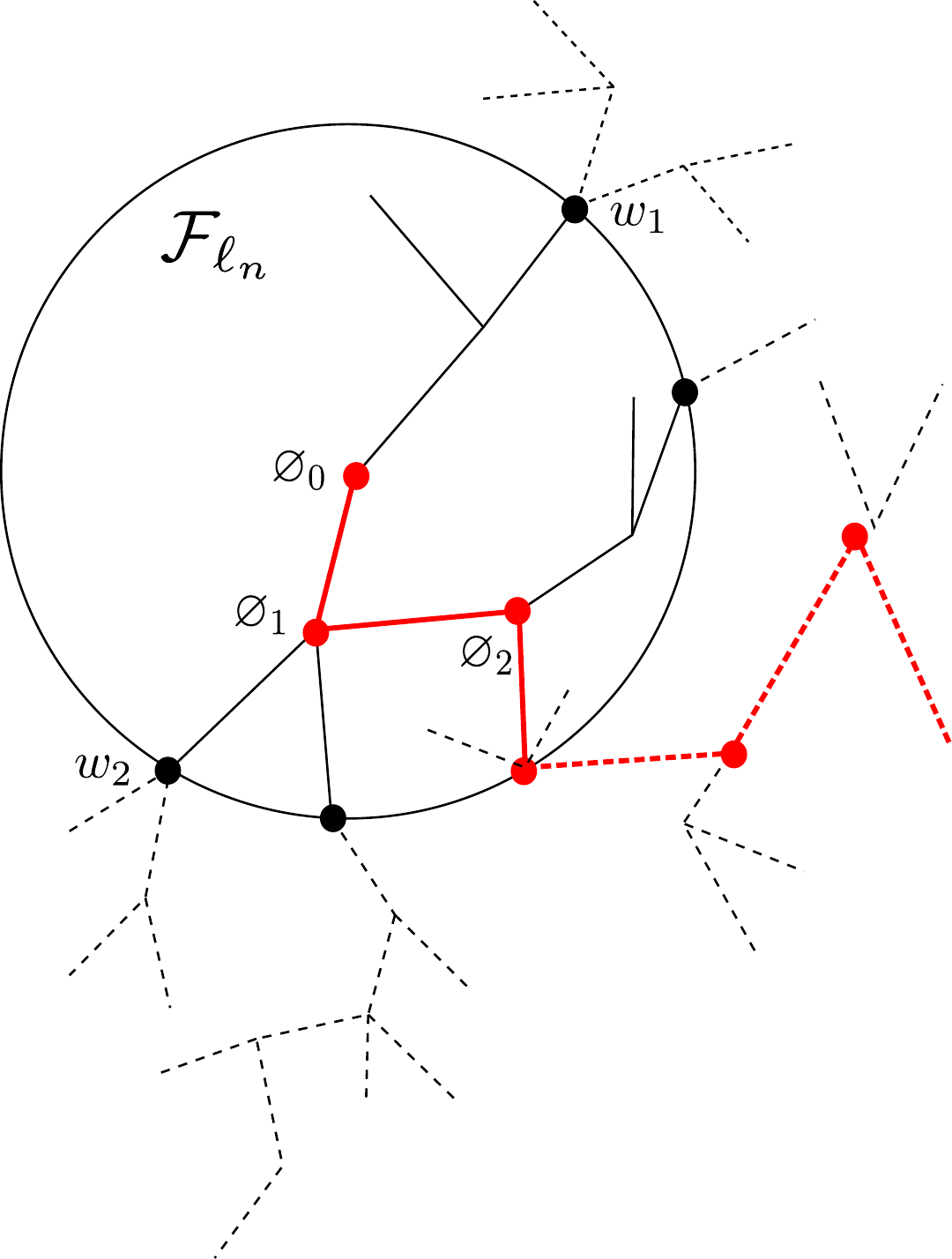}
	
	\caption{
		\leftskip=1.6truecm \rightskip=1.6truecm \small An illustration of the 
		BRW up to $\ell_n$. The spine is marked in red. Dotted lines are 
		independent of $\mathcal F_{\ell_n}$.}
\end{figure}

\begin{figure}[ht]
	\centering
	\includegraphics[scale=0.4]{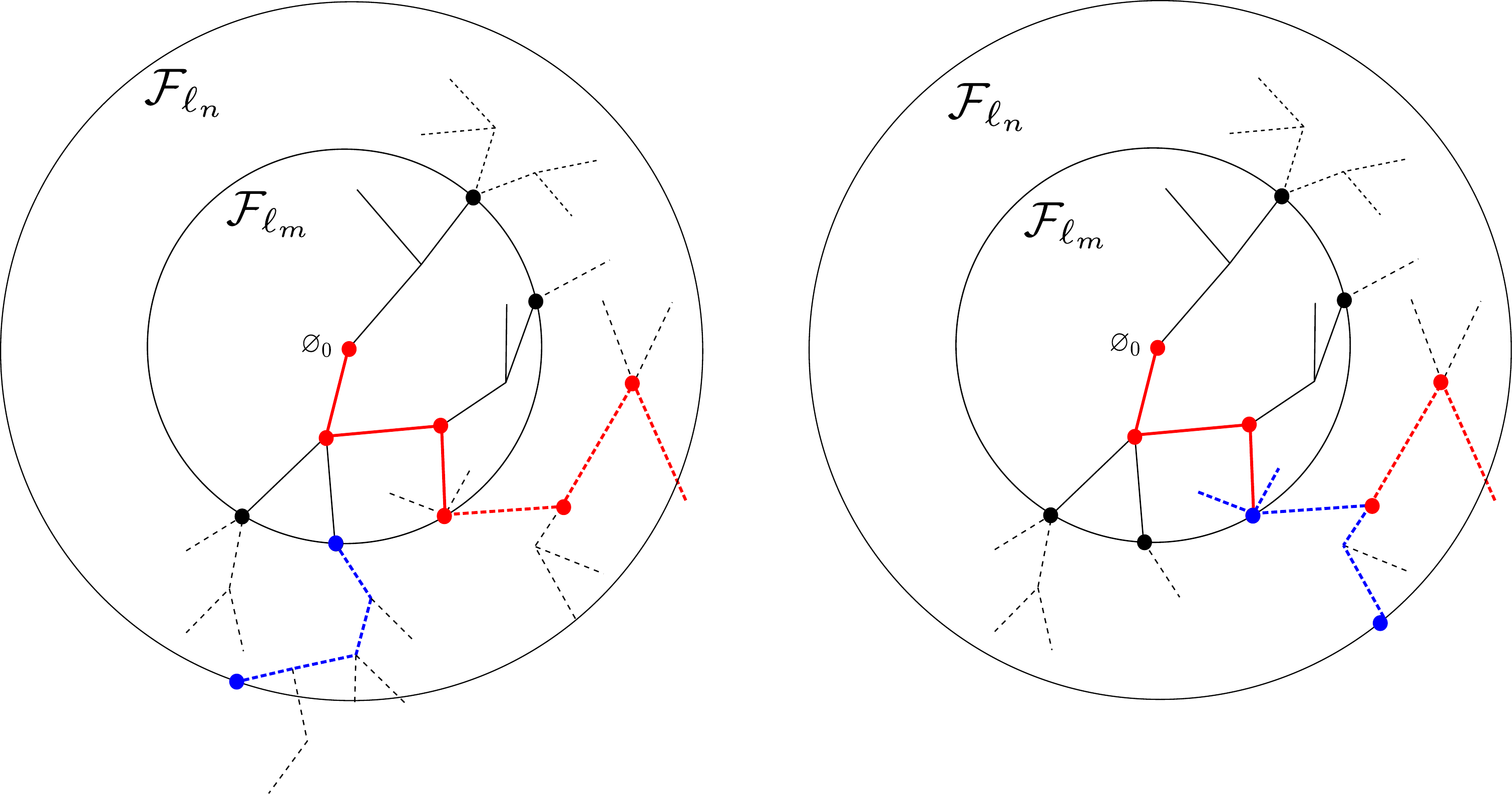}
	
	\caption{
		\leftskip=1.6truecm \rightskip=1.6truecm \small  Two illustrations for 
		$V[\varrho_{m}^{(n)},\varrho_{n}]$ (marked in blue). 
	}
\end{figure}

The following description of the law of $V[\varrho^{(n)}_m,\varrho_n]$ is therefore immediate:  

\begin{lemma}\label{lem:4.5} Let $1\le m< n$. Under $\Pbrw$, conditioned on 
	${\mathscr F}_{\ell_m}$, denote $\ell_m=\{w_1,\cdots,w_k\}$ in  
	lexicographical order  and $y_i=V(w_i)$ for $1\le i \le k$  (with $k=\#\ell_n\ge 1$).   Then there exists  some 
	positive (random, ${\mathscr F}_{\ell_m}$-measurable) numbers $(p_i)_{1\le i 
		\le k}$, such that $p_1+\cdots+p_k=1$ and 
	\begin{itemize}
		\item with probability $p_i\,(1\le i<k)$, $V[\varrho^{(n)}_m,\varrho_n]$ is 
		distributed as  $V[0, \varrho_n]$ under $\P_{y_i}^{(\tree)} 
		(\bullet 
		\, |\, \varrho_n < \infty)$, the 
		critical BRW started at $y_i$ and conditioned on exit from $ \ball(n)$; 
		\item with probability $p_k$, $V[\varrho^{(n)}_m,\varrho_n]$ is distributed as 
		$V[0, \varrho_n]$ under $\Pbrw_{y_k}$, the BRW indexed by 
		$\tree_\infty$ and started at $y_k$.  
	\end{itemize}
\end{lemma}

More specifically,  let for $1\le i \le k-1$, $A_i$ be the event that the BRW 
induced by the subtree  rooted at $w_i$, hits $\ball(n)$. 
Then $p_1:= \Pbrw(A_1|{\mathscr F}_{\ell_m}), p_2:= \Pbrw(A_2 \cap A_1^c|{\mathscr F}_{\ell_m}) = \Pbrw(A_2  |{\mathscr F}_{\ell_m}) \Pbrw( A_1^c|{\mathscr F}_{\ell_m}), ...,$
$p_{k-1}:= \Pbrw(A_{k-1} \cap A_1^c \cdots \cap A_{k-2}^c|{\mathscr F}_{\ell_m})$ 
, and $p_k:= 1- p_1-\cdots - p_{k-1}$.

\medskip
We end this subsection by a technical estimate  on the overshoot of $V$. Let \begin{equation}
E_n := \Big\{\max_{w \in \ell_n} |V(w)| \le  2n  \Big\}, \qquad F_n:=\Big\{  
V{[0,\varrho_n]} \subset \ball(2   n)\Big\}. \label{def-EN}
\end{equation}

\noindent
We show that  $E_n$ and $F_n$ hold  with overwhelming probability as $n \to \infty$:

\begin{lemma}\label{L:Enc} Assume \eqref{hyp-theta} and \eqref{hyp-GW}. Then  $$  \Pbrw(E_n^c) +  \Pbrw(F_n^c) \to 0, \qquad n \to\infty.$$
\end{lemma}

\begin{proof} For any $u\in \mathcal T_\infty$, let 
	$\Delta V(u):= V(u)- V({\buildrel \leftarrow \over  u})$
	be the displacement of $u$ with respect to its parent ${\buildrel \leftarrow 
		\over  u}$. Then 
	$$ E^c_n \subset\Big\{ \max_{w \in \ell_n} |\Delta V(w)| > n\Big\}, \qquad   
	F_n^c \subset \{|V(u_{\varrho_n})-V(u_{\varrho_n-1})|> n\Big\}.$$

	Let us prove at first $\P(F_n^c) \to 0$. For any $\varepsilon>0$, by \eqref{maxV<}, there is some   $C=C_\varepsilon>0$ such that for all large $n\ge n_0$, $\Pbrw(\varrho_n \le C n^4 ) \ge 1- \varepsilon$. Then $$ \Pbrw(F_n^c) \le 
	\Pbrw(|\Delta V(u_{\varrho_n})|> n)
	\le
	\varepsilon + \Pbrw(\max_{0\le i < C n^4}  |V(u_{i+1})-V(u_i)| > n) \to \varepsilon,$$
	
	\noindent as $n \to\infty$ by \eqref{max-V-increment1}.  This yields that $\limsup_{n\to \infty} \Pbrw(F_n^c) \le 
	\varepsilon$, hence is zero as $\varepsilon$ can be arbitrarily small. 
	
	To deal with $E^c_n$,  we observe that the spine intersects with $\ell_n$ at $\varnothing_J$ with $$J= \min\{j\ge 1: |X_j| \ge n\},$$ 
	
	\noindent where by a slight abus of notation, $X_j:=V(\varnothing_j), j\ge 0$ is a random walk on $\z^d$ with step distribution $\theta$. Then  \begin{eqnarray*} \Pbrw(|\Delta V(\varnothing_J)|> n) &=& \p(|X_J- X_{J-1}| > n)
		\\&=& \sum_{k=1}^\infty \p(\max_{1\le i \le k-1} |X_i| \le n, |X_k|> n, |X_k-X_{k-1}|> n) 
		\\
		&\le& \sum_{k=1}^\infty \p(\max_{1\le i \le k-1} |X_i| \le n) \p( |X_1|> n)
		\\
		&=& \e(J) \, \p( |X_1|> n).
	\end{eqnarray*}
	
	\noindent By the standard estimates for hitting time of a random walk,  $\e(J) \le C n^2$ so that $$\Pbrw(|\Delta V(\varnothing_J)|> n) \le C\, n^2  \, \p( |X_1|> n), \qquad \forall n\ge1.$$
	
	Now we deal with those $w \in \ell_n \backslash \{\varnothing_J\}$ such that $|\Delta V(w)|> n$. Each $w$  is the descendant of a tree rooted at $\varnothing_i$ for some $0\le i \le J-1$. For any $i\ge 0$, let $\kappa_i$ be  the number of subtrees rooted at
	$\varnothing_i$. Then $(\kappa_i)_{i\ge 0}$ are i.i.d. with 
	distribution $\Pbrw(\kappa_1=k)=\sum_{i=k+1}^\infty p_{i}.$ 
	We have $$ 1_{E_n^c} 
	\le 1_{\{|\Delta V(\varnothing_J)|\ge n\}} + 
	\sum_{k=1}^\infty 1_{\{J=k\}}   \sum_{i=0}^{k-1} \sum_{j=1}^{\kappa_i}\Theta^{(j)}_i, $$
	
	\noindent where for any $i$,  $\Theta^{(j)}_i, j\ge 1$ are i.i.d. and 
	distributed as $\sum_{w\in \ell_n(\tree)} 1_{\{| \Delta V(w)|\ge 
		n\}}$, under $\P_{V(\varnothing_i)}^{(\tree)}$,   and $\ell_n(\tree)$ is the 
	optional line defined from $\tree$ in the same way as $\ell_n$ does from $\tree_\infty$ [note that $\ell_n(\tree)$ may be empty]. Conditioned on $\{V(\varnothing_i)=x\}$,  
	the expectation of $\Theta^{(j)}_i$ is equal to    
	\begin{eqnarray*} 
		& & \sum_{m=0}^\infty \E_x^{(\tree)} \sum_{|w|=m} 1_{\{|V(w)| 
			>n, |\Delta V(w)|\ge n, \max_{u\in [\varnothing, w)} |V(u)|\le  n  \}}
		\\
		&=&
		\sum_{m=0}^\infty \p_x\Big( |X_m| >n, |X_m- X_{m-1}| \ge n,  \max_{0\le j\le m-1} |X_j|\le n\Big)
		\\
		&\le&
		\sum_{m=0}^\infty \p_x\Big( \max_{0\le j\le m-1} |X_j|\le n\Big) \p(|X|\ge n)
		\\
		&=&
		\p(|X|\ge n) \, \e_x (J+1), 
	\end{eqnarray*}
	
	\noindent where in the second equality we have used the fact that for each 
	$|w|=m$,  $V(u), u \in [\varnothing, w)$ is distributed as $X_j, 0\le j \le 
	m-1$. For the random walk $(X_j)$, again by the standard estimate on its 
	hitting time we have  $  \e_x (J) \le C\,  n^2, $  for all $ |x|\le n.$  It 
	follows that 
	\begin{eqnarray*} \Pbrw(E_n^c)
		&\le&
		C\, n^2  \, \p( |X|> n)+ \Ebrw(\kappa_0)\, \Ebrw \sum_{k=1}^\infty 1_{\{J=k\}} \, k\,  C  n^2 \p (|X|\ge n)  
		\\
		&=&
		C\, n^2  \, \p( |X|> n)+  C'   \e(J) \, n^2 \p(|X|\ge n)  
		\\
		&\le& C''  n^4  \p(|X|\ge n),
	\end{eqnarray*}
	
	\noindent  which converges to $0$ by the assumption \eqref{hyp-theta}. This completes the proof. \end{proof}

\subsection{Iteration by optional lines} \label{s:iteration}

For the SRW $(S_i)$, denote by $\tau_n$ its first exit time of $\ball(n)$: \begin{equation}\label{def-taun}\tau_n:= \min\{i\ge 0: |S_i| \ge n\}.\end{equation}

\begin{lemma}\label{lem:finite3} Assume \eqref{hyp-theta} and \eqref{hyp-GW}.
	In dimensions $d=3,4,5$, for every $\varepsilon>0$, there exists $\delta>0$ such that for all large  $n$, $$1_{E_n} \, \Pbrw(\Xi_n < \delta\, |\, {\mathscr F}_{\ell_n}) < \varepsilon,$$
	where $E_n$ is defined in \eqref{def-EN} and $$  \Xi_n:= \inf_{|x|\le n}\Pd_x(S[0,\tau_{6n}]\cap V[\varrho_{n}^{(6n)},\varrho_{6n}]\ne\emptyset) .
	$$\end{lemma}

\begin{proof}
	By \Cref{lem:4.5},  conditionally on ${\mathscr F}_{\ell_n}$, if $\ell_n=\{w_1,..., w_k\}$ and   $V(w_i)=y_i, 1\le i \le k$, then  
	$V[\varrho_{n}^{(6n)},\varrho_{6n}]$ is distributed as 
	
	$\bullet$ with probability $p_k$,  $V[0, \varrho_{6n}]$ under $\Pbrw_{y_k}$;
	
	$\bullet$ with probability $p_i$ for $1\le i < k$,   $V[0, \varrho_{6n}]$  under $\P^{(\tree)}_{y_i}(\bullet \,|\, \varrho_{6n} < \infty)$.

	\noindent Therefore,  $$\Pbrw(\Xi_n < \delta\, |\, {\mathscr F}_{\ell_n}) =\sum_{i=1}^{k-1} p_i \, a_n(y_i) + p_k  \, b_n(y_k),$$

	\noindent with 
	\begin{eqnarray*}
		a_n(y)&:=& \P^{(\tree)}_y \Big(\inf_{|x|\le n} \Pd_x (S[0,\tau_{6n}]\cap V[0, \varrho_{6n}]\ne\emptyset  ) < \delta  \,\big|\, \varrho_{6n} < \infty\Big)  , \\
		b_n(y)&:=& \Pbrw_y \Big(\inf_{|x|\le n}\Pd_x(S[0,\tau_{6n}]\cap V[0, \varrho_{6n}]\ne\emptyset) < \delta \Big).
	\end{eqnarray*}  On $E_n$, we have $\max_{1\le i \le k}|y_i| \le 2 n$, then $$ 1_{E_n} \Pbrw(\Xi_n < \delta\, |\, {\mathscr F}_{\ell_n}) \le \sup_{|y|\le 2 n} \, \max(    a_n(y) , b_n(y)).
	$$

	We deal with $\max_{|y|\le 2 n}b_n(y)$, and the argument can be easily adapted 
	to $\max_{|y|\le 2 n}a_n(y)$. We 
	shift $y$ to the origin, as the original structures of $V$ under $\Pbrw_y$ 
	exit  from $\ball(6n)$, the shifted versions at least exit from $\ball(4n)$.  
	Then  for all $|y|\le 2n$, $\inf_{|x|\le n}\Pd_x(S[0,\tau_{6n}]\cap V[0, 
	\varrho_{6n}]\ne\emptyset)$  under $\Pbrw_y$, is stochastically larger than 
	$\gamma_n$ under $\Pbrw$, where $$\gamma_n:= \inf_{|z|\le 3 n}  
	\Pd_z(S[0,\tau_{6n}]\cap V[0, \varrho_{4n}]\ne\emptyset).$$
	
	\noindent It follows that $$\max_{|y|\le 2 n}b_n(y) \le
	\Pbrw ( \gamma_n < \delta ).$$

	\noindent
	For any $\varepsilon\in (0, 1)$, let $\delta_1:=\delta_1(\varepsilon)>0$ be as in \Cref{lem:existproba} such that for all large $k$, under $\Pbrw$,   with probability at least $1-\frac\varepsilon4$,
	\begin{equation}\label{delta_1}
	{\inf_{|x|\le  k}\Pd_x(S[0, \infty)\cap V{[0,\varrho_{ k}]}\ne\emptyset) \ge \delta_1}.
	\end{equation}
	
	\noindent
	Now we choose $\alpha=\alpha(\delta_1)\in (0, 1)$ sufficiently small such that 
	\[\sup_{n\ge 1} \sup_{|x|\ge 6n}\Pd_x(\tau_{\ball(2 \alpha n)}<\infty)<\frac{\delta_1}{2}.\]

	\noindent Regardless of the BRW, we define  $$c_1(\alpha):= \inf_{n\ge n_0} \inf_{|z|\le 3n}\Pd_z( \tau_{\ball(2 \alpha n)} < \tau_{6n})>0,$$

	\noindent where $n_0=n_0(\alpha)$ is some large but fixed integer.   Let $n\ge n_0$. It follows that \begin{equation} \label{gammac1}
	\gamma_n  \ge  c_1 \, \inf_{|x|\le \alpha n}\Pd_x(S[0,\tau_{6n}]\cap V[0,\varrho_{\alpha n}] \ne\emptyset).
	\end{equation}

	\noindent Let as in \eqref{def-EN}   $$F_{\alpha n}:=\Big\{ V{[0,\varrho_{\alpha n}]} \subset \ball(2 \alpha n)\Big\}.$$ 
	
	\noindent  By Lemma \ref{L:Enc},  $\Pbrw(F_{\alpha n}^c) \le  \frac\varepsilon4$ for all $n\ge n_0$ (we may enlarge $n_0$ if necessary).    On $F_{\alpha n}$, $$\Pd_x(S[0,\tau_{6n}]\cap V{[0,\varrho_{\alpha n}]}\ne\emptyset) \ge \Pd_x(S[0,\infty) \cap V{[0,\varrho_{\alpha n}]}\ne\emptyset) -\frac{\delta_1}{2},$$

	\noindent which in view of \eqref{delta_1} is larger than $\frac{\delta_1}2$ with probability at least $1- \frac\varepsilon4$. It follows from \eqref{gammac1} that  under $\Pbrw$,  with probability at least $1-\frac\varepsilon2$,  $$\gamma_n \ge   c_1 \delta_1 /2.$$
	
	\noindent This means that if $\delta< c_1 \delta_1 /2$, then $\max_{|y|\le 2 n}b_n(y)  \le   \frac\varepsilon2$. We may treat $\max_{|y|\le 2 n} a_n(y) $ in the same way by using Corollary \ref{cor:4.4} and obtain the Lemma. 
\end{proof}

\begin{lemma}\label{lem:epsilonMd} Assume \eqref{hyp-theta} and \eqref{hyp-GW}.
	In dimensions $d=3,4,5$, let
	\[
	c(\lambda,n):=\sup_{|x|<\lambda n}\Pd_x\pars*{S[0,\tau_{n}]\cap V[0,\varrho_{n}]=\emptyset},
	\]
	then for any $M>0$, there exists $\upsilon>0$ such that for any $\lambda\in(0,1)$ small enough,
	\[
	\limsup_{n\rightarrow \infty}\Pbrw_0(c(\lambda,n)>\lambda^\upsilon)<\lambda^M.
	\]
\end{lemma}
\begin{proof}  Let $\lambda \in (0, 1)$ be small.  Set $K:=\lfloor\log_6 \lambda^{-1}\rfloor$ and $m=\lambda n$. Define for $k=0,1,\cdots,K-1$, \[
	g(m,k):=\inf_{|x|\le  m 6^k}\Pd_x\pars{S[0,\tau_{m 6^{k+1}}]\cap V[\varrho_{m 6^k}^{(m 6^{k+1})},\varrho_{m 6^{k+1}}]\ne\emptyset}.
	\]
	
	\noindent Then $$c(\lambda,n)\le \prod_{k=0}^{K-1} (1-g(m,k)) .$$

	Let   $$E_k:= \Big\{\max_{w \in \ell_{m 6^k}} |V(w)| \le  2 m 6^k \Big\}, 
	\qquad 0 \le k \le K-1.$$
	
	\noindent Note that $E_k$ is measurable with respect to ${\mathscr F}_{\ell_{m 6^k}}$. Let $\varepsilon\in (0, 1)$ be some small constant whose value will be chosen later.  By  \Cref{lem:finite3}, there is some small $\delta>0$ such that for all large $n\ge n_0(\varepsilon, \delta, \lambda)$,  \begin{equation}\label{Ekgmk} 1_{E_k} \, \Pbrw(g(m, k) < \delta \, | \, {\mathscr F}_{\ell_{m 6^k}}) < \varepsilon, \qquad \forall\, 0\le k \le K-1. \end{equation}

	On the event $\{ \sum_{k=0}^{K-1} 1_{\{g(m, k) \ge \delta \}} > \frac{K}2\}$, $ c(\lambda, n) \le (1-\delta)^{K/2}\sim \lambda^{- \frac12 \log_6 (1-\delta)}.$  Then if we take $\upsilon:= - \frac13\, \log_6 (1-\delta)$, we have \begin{eqnarray*} \Pd(c(\lambda, n) > \lambda^\upsilon) 
		&\le & \Pbrw\Big(  \sum_{k=0}^{K-1} 1_{\{g(m, k) \ge \delta \}}  \le \frac{K}{2}\Big)
		\\
		&\le&
		\Pbrw(\cup_{k=0}^{K-1} E_k^c) + \Pbrw\Big( \sum_{k=0}^{K-1} 1_{E_k\cap \{g(m, k) < \delta \}}  > \frac{K}{2}\Big).
	\end{eqnarray*}
	
	For the first term we use the union bound: $$\Pbrw(\cup_{k=0}^{K-1} E_k^c)
	\le
	\sum_{k=0}^{K-1} \Pbrw( E_k^c),   $$
	
	\noindent which, according to Lemma \ref{L:Enc}, converges to $0$ as $n\to\infty$.

	For the second term we use the Chebyshev inequality: for any $s>0$, 
	$$\Pbrw\pars*{ \sum_{k=0}^{K-1} 1_{E_k\cap \{g(m, k) < \delta \}}  > 
		\frac{K}{2}}
	\le
	e^{- s K/2} \, \Ebrw \bracks*{ \prod_{k=0}^{K-1} e^{s 1_{E_k\cap \{g(m, k) < 
				\delta \}} }}.$$
	
	\noindent By using  \eqref{Ekgmk},    
	$$\Ebrw \bracksof*{   e^{s 1_{E_k\cap \{g(m, k) < \delta \}} }}{ 
		{\mathscr F}_{\ell_{m 6^k}} } = 1+ ( e^s-1) \Pbrw \parsof*{    E_k\cap \{g(m, 
		k) < \delta \} } { {\mathscr F}_{\ell_{m 6^k}}} \le 1 + (e^s -1) 
	\varepsilon.$$ 
	
	\noindent Using these inequalities successively for $k=K-1, K-2, ..., 0$, we see that $$
	\Ebrw \bracks*{ \prod_{k=0}^{K-1} e^{s 1_{E_k\cap \{g(m, k) < \delta \}} }} 
	\le
	(1+(e^s -1) \varepsilon) ^K 
	\le
	e^{K (e^s -1) \varepsilon}< e^{K e^s \, \varepsilon}.$$
	
	\noindent Now for any $M>0$, we may find some $s>0$ large enough such that 
	$e^{- s K/4}\le \lambda^M$.  Then we choose and fix $\varepsilon$ small enough 
	such that $  \varepsilon \le \frac{s}{4}e^{-s}$. It follows that 
	$$\Pbrw\pars*{ \sum_{k=0}^{K-1} 1_{E_k\cap \{g(m, k) < \delta \}}  \ge 
		\frac{K}{2}}
	\le e^{- s K/2+ K e^s \varepsilon} \le 
	e^{-s K/4} \le \lambda^M,$$
	
	\noindent ending the proof.  
\end{proof}

We need an analogue of \Cref{lem:epsilonMd} for fixed time in place of 
stopping time: Let for $n\ge 1$ and $0< \lambda <1$,  $$\widehat  c(\lambda, 
n):= \sup_{|x|<\lambda n } \Pd_x(\tau_{V[0, n^4 ]}=\infty).$$

\begin{lemma}\label{lem:epsilonMd2} Assume \eqref{hyp-theta} and \eqref{hyp-GW}.
In dimensions $d=3,4,5$, for any $M>0$, there exists $\upsilon>0$ such that for any $\lambda\in(0,1)$ small enough,
\[
\limsup_{n\rightarrow \infty}\Pbrw(\widehat 
c(\lambda,n)>\lambda^\upsilon)<\lambda^M.
\]
\end{lemma}

\begin{proof} By Lemma \ref{L:maxV},     
$$ \limsup_{n\to \infty}\Pbrw\Big( 
\max_{0\le i \le n^4} |V(u_i)| \le \lambda^{1/2}   n \Big) \le C' \, e^{- C 
\lambda^{-1/2}}.$$ 

On $\{\max_{0\le i\le n^4} |V(u_i)| > \lambda^{1/2} n\}$, 
\[
\widehat c(\lambda, n) 
\le \sup_{|x|<\lambda n } \Pd_x\pars*{\tau_{V[0, 
\varrho_{\lambda^{1/2} n}  ]}=\infty}
\le c(\lambda^{1/2},\lambda^{1/2}n),
\]
thus
\begin{align*}
&\Pbrw(\widehat c(\lambda,n)>\lambda^\upsilon) \\
\le 
&\Pbrw \pars*{ \max_{0\le i \le n^4} |V(u_i)| \le \lambda^{1/2}   n  }
+ 
\Pbrw\pars*{
c(\lambda^{1/2}, \lambda^{1/2} n)> \lambda^\upsilon}\\
\le&
C' \, e^{- C \lambda^{-1/2}}
+ 
\Pbrw\pars*{
c(\lambda^{1/2}, \lambda^{1/2} n) > \lambda^\upsilon}.
\end{align*}

The result follows by taking $\lambda$ small enough and then applying Lemma 
\ref{lem:epsilonMd}. 
\end{proof}

To prove \Cref{thm:mainfinite}, we need the help of $\tree_\infty^*$ again, 
which requires an analogue of \Cref{lem:epsilonMd2} for 
$\tree_\infty^*$. This, however, is nontrivial.
Indeed, the main difference of the two models is the spine, whose spatial positions are given by a SRW. But  two 
independent SRWs up to time $n^4$ with starting points $O(n)$ distance apart intersect 
with a positive probability in dimension $d=3$. To avoid this issue, we use the projection trick in the proof of \Cref{lem:existproba} again. 
Our strategy is to prove \Cref{thm:mainfinite} for dimension $d=5$ first, based 
on the following corollary of \Cref{lem:epsilonMd2}, and then apply the 
projection trick.
\begin{corollary} \label{cor:epsilonMd2-new} Assume \eqref{hyp-theta} and 
\eqref{hyp-GW}.
	Let for $n\ge 1$ and $0< \lambda <1$, 
	$$c^\new(\lambda, n):= 
	\sup_{\substack{|x|<\lambda n}} 
	\Pd_x(\tau_{V^*[0, 
	n^4]}=\infty).$$
	In dimensions $d=5$, for any $M>0$, there exist  $\upsilon>0$ such that 
	for any $\lambda\in(0,1)$,
	\begin{equation} \label{c*new}
	\limsup_{n\rightarrow \infty}\mathbf 
	P^{(\tree_\infty^*)}(c^\new(\lambda,n)>\lambda^\upsilon)<\lambda^M.
\end{equation} 
\end{corollary}
\begin{proof}[Proof of Corollary \ref{cor:epsilonMd2-new}]
Let
$\tau_{n}:= \min\{i\ge 0: S_i \not \in  {\ball}(n)\}$ be the first exit time of $S$ from ${\ball}(n)$. 
By 
\cite[Lemma 6.3.7]{Lawler-Limic}, there exists some positive constant $c$ such that for all large $k$, uniformly in $|x| \le k/4$ and $y\in \partial{\ball(k)}$, $\Pd_x(S_{\tau_{{k}}}=y) \le c \, \Pd_0(S_{\tau_{{k}}}=y).$

\noindent It follows that for all $\lambda\in(0,1/4)$, $n$ large enough   and $|x| \le \lambda n$, we have  
\begin{equation}
\begin{aligned} 
\Pd_x(\tau_{V^*[0, n^4]}=\infty)
\le&
\sum_{y \in \partial{\ball(4\lambda n)}} \Pd_x(S_{\tau_{4\lambda n}}=y) \Pd_y(\tau_{V^*[0, n^4]}=\infty)
\\
\le&
c\, \sum_{y \in \partial\ball(4\lambda n)} \Pd_0(S_{\tau_{4\lambda n}}=y) \Pd_y(\tau_{V^*[0, n^4]}=\infty).
\label{eq:harmonic}
\end{aligned}
\end{equation}

By \eqref{eq:independentX1}, for any $y\in \z^d$, 
\begin{equation*}
\begin{aligned}
\Pd_{y}(\tau_{V^*[0, n^4]}=\infty)
\le&   1_{\braces*{{\tt t}_0^* \ge \frac{n^4}{2}}}+ 
	\Pd_{y}(\tau_{V[1,  n^4/2]+X}=\infty)
	+\Pd_{y}(\tau_{\{V(\varnothing_i), i\ge 1\}+X}<\infty)
\\
\le&
  1_{\braces*{{\tt t}_0^* \ge \frac{n^4}{2}}}
  +	\Pd_{y}(\tau_{V[1,  n^4/2]+X}=\infty)
  + \sum_{i=1}^\infty G^{(d)}(V(\varnothing_i)+X-y),
\end{aligned}
\end{equation*}

\noindent where in the last inequality we use the fact that for any $z\in \z^d$, 
$\Pd_y(\exists j\ge 0: S_j=z) \le \sum_{j=0}^\infty \Pd_y(S_j=z)= G^{(d)}(z-y).$  
Put this into \eqref{eq:harmonic}, we have that  
\begin{equation}\label{eq:c_decomposition}
\begin{aligned}   
c^\new(\lambda, n)
\le&
c \, 1_{\braces*{{\tt t}_0^* \ge \frac{n^4}{2}}}+ 
c \, 1_{\braces*{|X| \ge {\lambda n}}}
+ c\max_{|y|\le  5\lambda n} \Pd_y(\tau_{V[1,n^4/2]}=\infty) 
\\
&+ c\, \sum_{y \in \partial{\ball(4\lambda n)}} \Pd_0(S_{\tau_{\partial{\ball}{(4\lambda n)}}}=y)  \sum_{i=1}^\infty G^{(d)}(V(\varnothing_i)+X-y)
\\
=:& c^\new_{1}(\lambda, n)+c^\new_{2}(\lambda, n)+c^\new_{3}(\lambda, n)+c^\new_{4}(\lambda, n).
\end{aligned}
\end{equation}

As $n\rightarrow\infty$,
\begin{equation}\label{eq:c1new}
\begin{aligned} 	
\mathbf 
	P^{(\tree_\infty^*)}(c^\new_{1}(\lambda,n)+c^\new_{2}(\lambda,n)>0)
\le
	\mathbf P^{(\tree_\infty^*)} ( {\tt t}_0^* \ge {n^4}/{2}) + \mathbf P^{(\tree_\infty^*)} (|X| \ge {\lambda n})
\to 0.
\end{aligned}
\end{equation}

Note that $c^\new_{3}(\lambda,n)\le  \widehat c(\lambda', n')=\max_{|y|\le   \lambda' n'} \Pd_y(\tau_{V[1,(n')^4]}=\infty)$ with $\lambda':= 10   \lambda$, $n':= \lfloor 2^{-1/4} n\rfloor$. By   \Cref{lem:epsilonMd2},  for any $M>0$, there exists $\upsilon'>0$ such that for all small $\lambda'$, $\limsup_{n'\rightarrow \infty}\mathbf 
	P^{(\tree_\infty^*)}(\widehat c(\lambda', n')>(\lambda')^{\upsilon'})<(\lambda')^{M+1}$. Let $\upsilon:=\upsilon'/2$, then for all small $\lambda>0$,  
\begin{equation}\label{eq:c3new} 	
\limsup_{n\rightarrow \infty}\mathbf 
	P^{(\tree_\infty^*)}(c^\new_{3}(\lambda,n)>\lambda^\upsilon/2)<\lambda^M. 
\end{equation}

For $c^\new_{4}(\lambda, n)$,  note that under $\P^{(\tree^*_\infty)}$, $(V(\varnothing_i)+X)_{i\ge 1}$ is distributed as $(X_{i+1})_{i\ge 1}$,  a random walk on $\z^d$ with step distribution $\theta$. Let $d=5$.   There is some constant $c'>0$ such that  
\begin{align*}
\E^{(\tree^*_\infty)} \Big( \sum_{i=1}^\infty G^{(d)}(V(\varnothing_i)+X-y)\Big) 
\le  \sum_{z\in \z^d}  G^{(d)}_\theta(z) \,  G^{(d)}(z -y) \le   c' (1+|y|)^{-1},
\end{align*}
where 
$
G^{(d)}_\theta (z):= \sum_{n=0}^\infty \p(X_n=z)
$, and we cite \cite[(1.5a)]{uchiyama1998green} for its asymptotic.
Hence
$\E^{(\tree^*_\infty)}  (c^\new_{4}(\lambda, n))
\le
c \, c'  \,  (1+|4\lambda n|)^{-1}, $
which implies that $$\limsup_{n\rightarrow \infty}\mathbf 
	P^{(\tree_\infty^*)}(c^\new_{4}(\lambda,n)>\lambda^\upsilon/2)=0.$$ This together with  \eqref{eq:c_decomposition}, \eqref{eq:c1new}, and \eqref{eq:c3new} yield Corollary \ref{cor:epsilonMd2-new}.
\end{proof}

\begin{proof}[Proof of \Cref{thm:mainfinite}] As explained in the proof of 
	Lemma \ref{lem:existproba}, it suffices to show the Theorem for $d=5$, because 
	the result for dimensions $d=3,4$ follows by using the projection
	$$
	(x_1,x_2,x_3,x_4,x_5)\mapsto
	\left\{
	\begin{array}{ll}
	(x_1,x_2,{x_3+x_4+x_5}), & d=3\\
	(x_1,x_2,x_3,{x_4+x_5}), & d=4
	\end{array}
	\right..
	$$  
	
	Let $d=5$. We focus on  the model $\tree_\infty^*$ first.
	
	Fix $0< \zeta< \frac14 - \frac1{q}$. Let $\varepsilon:=\lambda^{1/\zeta}$. 
	By \eqref{V-increment}, there is some constant $a>0$ such that 
	$$ \mathbf P^{(\tree_\infty^*)}\big( E_{n,\varepsilon}\big) \le C\, \lambda^{a/\zeta},$$
	
	\noindent with $$E_{n,\varepsilon}:=\Big\{ \max_{0\le k \le \frac1\varepsilon} 
	\max_{0\le j \le \varepsilon n} | V(u_{j+ k \varepsilon n}^*)- V(u_{k 
		\varepsilon n}^*)| \ge \lambda\,  n^{1/4}\Big\}.$$
	
	On $E^c_{n, \varepsilon}$,  for any $x$ such 
	that $\dist(x,V^*[0,n])<\lambda n^{1/4}\}$, there exists some $0\le k \le 
	\frac1{\varepsilon}$ such that $\dist(x, V(u_{k \varepsilon n}^*))<2\lambda 
	n^{1/4}\}$. It follows that on $E^c_{n, \varepsilon}$,
	\begin{eqnarray*}
		\Theta_n &:=& \sup_{\dist(x,V^*[0,n])<\lambda n^{1/4}}\Pd_x(\tau_{V^*[0, 3n/2 
			]}=\infty)\\
		&\le&
		\max_{0\le k \le \frac1{\varepsilon}}
		\sup_{
			{\dist(x, V(u_{k \varepsilon n}^*))<2\lambda n^{1/4}}
		}
		\Pd_x(\tau_{V^*[ k \varepsilon n , k \varepsilon n +n/2]}=\infty).
	\end{eqnarray*}
	
	\noindent
	By \eqref{invariance}, each $V^*[ k \varepsilon n , k \varepsilon n +n/2]$, shifted by $V(u_{k \varepsilon n}^*)$,  is 
	distributed as $V^*[0,n/2]$. Therefore the union 
	bound yields that   
	\begin{eqnarray*}
		\mathbf P^{(\tree_\infty^*)}   ( \Theta_n \ge \lambda^\upsilon , E^c_{n, \varepsilon}) 
		&\le& (1+ \frac1{\varepsilon}) \mathbf P^{(\tree_\infty^*)} 
		\Big(
		\sup_{ |x|<2\lambda n^{1/4}}
		\Pd_x(\tau_{V^*[0, n/2]}=\infty) > 
		\lambda^\upsilon\Big)
		\\
		&\le&
		(1+ \frac1{\varepsilon})   \lambda^{M} ,
	\end{eqnarray*}

	\noindent for all large $n$, where for the last inequality we have applied \Cref{cor:epsilonMd2-new} to an arbitrary constant $M> \frac1\zeta$ and the corresponding $\upsilon>0$.  Then   $$\limsup_{n\to\infty}\mathbf E^{(\tree_\infty^*)} (\Theta_n)\le  \lambda^\upsilon+  C\, \lambda^{a/\zeta}+  (1+ 
	\lambda^{-1/\zeta})   \lambda^{M} \to 0 , \qquad \lambda\to0,$$
	
	\noindent proving the Theorem for $\tree_\infty^*$. 
	
	To deal with $\tree_\infty$, we apply \eqref{coupling2trees} and obtain that 
	for any fixed $\delta>0$,  under $\P^{(\tree^*_\infty)}(\bullet \, |\, 
	{\tt t}_0^*=0)$, with probability $1-o(1)$, $V[0,n]$ contains $V^*[1,(1-\delta)n]$. 
	Therefore the conclusion  \eqref{SV} for $V[0,n]$ follows from that of 
	$V^*[1,(1-\delta)n]$ under the event $\{{\tt t}_0^*=0\}$. 
\end{proof}

\subsection{Intersection probabilities: Proof of Theorem \ref{lem:inter1}}\label{sec:3.2}

We are entitled to give the proof of Theorem \ref{lem:inter1}:

\begin{proof}[Proof of Theorem \ref{lem:inter1}]
	It suffices to compare $R_n$ under $\Pmodel$ to $V[0, n]$  under $\Pbrw$ in \Cref{thm:mainfinite}, which follows from the arguments in \cite[Section 5]{Zhu-cbrw} for the coupling between the two models.
	
	Indeed, write $R_n[0,k]$ for the first $k+1$ positions in $R_n$ in lexicographical order, then \begin{eqnarray*} \Emodel \Big[
		\sup_{\dist(x,R_n[0,n/2])<\lambda n^{1/4}}\Pd_x(\tau_{R_n}=\infty) \Big]
		&\le&
		\Emodel  \Big[
		\sup_{\dist(x,R_n[0,n/2])<\lambda n^{1/4}}\Pd_x(\tau_{R_n([0, 3n/4])}=\infty)\Big]
		\\
		&\le&
		C\,  \Ebrw\Big[
		\sup_{\dist(x,V[0,n/2])<\lambda n^{1/4}}\Pd_x(\tau_{V[0, 3n/4]}=\infty)\Big],
	\end{eqnarray*}
	
	\noindent where the last inequality is due to \eqref{absolutecontinuity}.  Then by  \Cref{thm:mainfinite}, $$ \limsup_{n\to\infty} \Emodel \Big[
	\sup_{\dist(x,R_n[0,n/2])<\lambda n^{1/4}}\Pd_x(\tau_{R_n}=\infty) \Big] \to 0, \qquad \lambda\to 0.$$
	
	\noindent Since the other half 
	\[
	\sup_{\dist(x,R_n[n/2,n])<\lambda n^{1/4}}\Pd_x(\tau_{R_n}=\infty)
	\]
	can be treated in the same way, the conclusion follows. 
\end{proof}

\medskip
{\noindent\bf Acknowledgements.}
 The authors would like to thank Jean-Fran\c cois Delmas for helpful discussions on ISE.

\end{document}